\documentclass[11pt,a4paper]{article}
\usepackage[utf8]{inputenc}
\usepackage[T1]{fontenc}

\usepackage{amsmath}
\usepackage{amsthm}
\usepackage{amssymb}
\usepackage{amsopn}
\usepackage{enumitem}
\usepackage{color}
\usepackage{mathtools}
\usepackage{hyperref}
\usepackage{doi}

\usepackage[margin=3cm]{geometry}

\theoremstyle{plain}
\newtheorem{theorem}{Theorem}[section]
\newtheorem{proposition}[theorem]{Proposition}

\newtheorem{lemma}[theorem]{Lemma}

\theoremstyle{definition}

\newtheorem{example}[theorem]{Example}

\theoremstyle{remark}
\newtheorem{remark}[theorem]{Remark}

\DeclareMathOperator{\Id}{Id}
\DeclareMathOperator{\id}{id}
\DeclareMathOperator{\sym}{sym}

\DeclareMathOperator{\skewo}{skew}

\DeclareMathOperator{\dist}{dist}

\DeclareMathOperator{\SO}{SO}
\DeclareMathOperator{\Oo}{O}

\DeclareMathOperator*{\argmin}{arg\,min}

\def\Xint#1{\mathchoice
   {\XXint\displaystyle\textstyle{#1}}%
   {\XXint\textstyle\scriptstyle{#1}}%
   {\XXint\scriptstyle\scriptscriptstyle{#1}}%
   {\XXint\scriptscriptstyle\scriptscriptstyle{#1}}%
   \!\int}
\def\XXint#1#2#3{{\setbox0=\hbox{$#1{#2#3}{\int}$}
     \vcenter{\hbox{$#2#3$}}\kern-.5\wd0}}

\def\dashint{\Xint-}

\newcommand{\Z}{\mathbb{Z}}

\newcommand{\N}{\mathbb{N}}
\newcommand{\R}{\mathbb{R}}
\newcommand{\wto}{\rightharpoonup}

\newcommand{\eps}{\varepsilon}
\newcommand{\dtilde}[1]{\tilde{\raisebox{0pt}[0.85\height]{$\tilde{#1}$}}}
\newcommand{\dbar}[1]{\bar{\raisebox{0pt}[0.95\height]{$\bar{#1}$}}}

\usepackage{color}

\newcommand{\e}{{\mathrm e}}
\makeatletter
\renewcommand*\env@matrix[1][*\c@MaxMatrixCols c]{%
  \hskip -\arraycolsep
  \let\@ifnextchar\new@ifnextchar
  \array{#1}}
\makeatother

\begin{document}
\begin{center}
\begin{Large}
An atomistic derivation of von-Kármán plate theory
\end{Large}
\\[0.5cm]
\begin{large}
Julian Braun\footnote{University of Warwick, UK, {\tt j.braun@warwick.ac.uk}} and Bernd Schmidt\footnote{Universität Augsburg, Germany, {\tt bernd.schmidt@math.uni-augsburg.de}}
\end{large}
\\[0.5cm]
\today
\\[1cm]
\end{center}

\begin{abstract}
We derive \emph{von-Kármán plate theory} from three dimensional, purely atomistic models with classical particle interaction. This derivation is established as a $\Gamma$-limit when considering the limit where the interatomic distance $\eps$ as well as the thickness of the plate $h$ tend to zero. In particular, our analysis includes the \emph{ultrathin} case where $\eps \sim h$, leading to a new \emph{von-Kármán plate theory for finitely many layers}. 
\end{abstract}

\section{Introduction}

The aim of this work is to derive von-Kármán plate theory from nonlinear, three-di\-men\-sion\-al, atomistic models in a certain energy scaling as the interatomic distance $\eps$ and the thickness of the material $h$ both tend to zero.

The passage from atomistic interaction models to continuum mechanics (i.e., the limit $\eps \to 0$) has been an active area of research over the last years. In particular, this limit has been well studied for three-dimensional elasticity, cf., e.g.,\ \cite{BLL:02, alicandrocicalese, schmidtlinelast, braunschmidt13, emingstatic, ortnertheil13, BraunSchmidt16, Braun17}. At the same time, there have emerged rigorous results deriving effective thin film theories from three-dimensional nonlinear (continuum) elasticity in the limit of vanishing aspect ratio (i.e., the limit $h \to 0$), cf.\ \cite{LeDretRaoult:95,FJM:02,FJM:06,ContiMaggi2008,olbermannruna17}. First efforts to combine these passages and investigate the simultaneous limits $\eps\to0$ and $h\to0$ were made in \cite{FJ:00,schmidt-membrane,Sch:08b} for membranes (whose energy scales as the thickness $h$) and in \cite{Sch:06} for Kirchhoff plates (whose energy scales like $h^3$). In particular, this left open the derivation of the von-Kármán plate theory, which describes plates subject to small deflections with energy scale $h^5$ and might even be the most widely used model for thin structures in engineering. Though we do want to mention \cite{bartels17} for a result regarding discrete von-Kármán plate theory that is motivated numerically and not physically.

Our first aim is to close this gap. For thin films consisting of many atomic layers one expects the scales $\eps$ and $h$ to separate so that the limit $\eps,h\to0$ along $\frac{h}{\eps} \to \infty$ is equivalent to first passing to the continuum limit $\eps\to0$ and reducing the dimension from 3d to 2d in the limit $h\to0$. We will show in Theorem~\ref{thm:Gammalimit}a) that this is indeed true.

By way of contrast, for {\em ultrathin} films consisting of only a few atomic layers, more pecisely, if $\eps,h\to0$ such that the number of layers $\nu = \frac{h}{\eps}+1$ remains bounded, the classical von-Kármán theory turns out to capture the energy only to leading order in $\frac{1}{\nu}$. The next aim is thus to derive a new finite layer version of the von-Kármán plate theory featuring additional explicit correction terms, see Theorem~\ref{thm:Gammalimit}b). In view of the fabrication of extremely thin layers, such an analysis might be of some interest also in engineering applications. An interesting question related to such applications, which we do not address here, would be to extend our analysis to heterogeneous structures as in \cite{DeBenitoSchmidt:19a,DeBenitoSchmidt:19b}. 

Our third aim concerns a more fundamental modelling point of view which is based on the very low energy of the von-Kármán scaling: If the the plate is not too thick (more precisely, if $\frac{h^5}{\eps^3} \to 0$), we strengthen the previous results to allow for a much wider range of interaction models, that allow for much more physically realistic atomic interactions (compared to \cite{FJM:02,FJM:06}) as they can now be invariant under reflections and no longer need to satisfy growth assumptions at infinity, see Theorems~\ref{thm:Gammalimit2} and \ref{thm:Gammalimit3}. In particular, this includes Lennard-Jones-type interaction models, see Example~\ref{ex:mass-spring-LJ}. 

Finally, on a technical note, the proof of the our main result set forth in Section~\ref{section:Proofs} elucidates the appearance and structure of the correction terms in the ultrathin film regime. Both in \cite{Sch:06} and the present contribution, at the core of the proof lies the identification of the limiting strain, which in the discrete setting can be seen as a $3 \times 8$ matrix rather than a $3 \times 3$ matrix. In \cite{Sch:06} this has been accomplished with the help of adhoc techniques that allowed to compare adjacent lattice unit cells. Now, for the proof of Proposition~\ref{prop:limiting-strain} we introduce a more general and flexible scheme to capture discreteness effects by splitting the deformation of a typical lattice unit cell into affine and non-affine contributions and passing to weak limits of taylor-made finite difference operators. While for $h \gg \eps$ these operators will tend to a differential operator in the limit, if $h \sim \eps$, finite differences in the $x_3$ direction will not become infinitestimal and lead to lower order corrections in $\frac{1}{\nu}$. 
 
This work is organized as follows: In Section \ref{section:Models-and-Results}, we first describe the atomistic interaction model and then present our results. Our main theorem, Theorem \ref{thm:Gammalimit}, details the $\Gamma$-limits for both the \emph{thin} ($\nu \to \infty$) and \emph{ultrathin} ($\nu$ bounded) case. Theorems \ref{thm:Gammalimit2} and \ref{thm:Gammalimit3} then extend these results to more general and more physically realistic models. Section \ref{sec:preparations} contains a few technical tools to circumvent rigidity problems at the boundary and to compare continuous with discrete quantities. Using these tools we then prove our results in Section \ref{section:Proofs}.


\section{Models and Results}\label{section:Models-and-Results}


\subsection{Atomic Model}

Let $S \subset \R^2 = \R^2 \times \{ 0 \} \subset \R^3$ be an open, bounded, connected, nonempty set with Lipschitz boundary. To keep the notation simple we will only consider the cubic lattice. Let $\eps>0$ be a small parameter describing the interatomic distance, then we consider the lattice $\eps \Z^3$. We denote the number of atom layers in the film by $\nu \in \N$, $\nu \geq 2$ and the thickness of the film by $h=(\nu-1) \eps$. In the following let us consider sequences $h_n, \eps_n, \nu_n$, $n \in \N$, such that $\eps_n, h_n \to 0$. The macroscopic reference region is $\Omega_n = S \times (0,h_n)$ and so the (reference) atoms of the film are $\Lambda_n = \overline{\Omega}_n \cap \eps_n \Z^3$. We will assume that the energy can be written as a sum of cell energies. 

More precisely, as in \cite{Sch:06} we let $z^1, \dots, z^8$ be the corners of the unit cube centered at $0$ and write
\[Z=(z^1, \dots, z^8)= \frac{1}{2}\begin{pmatrix*}[r]
-1&1&1&-1&-1&1&1&-1\\-1&-1&1&1&-1&-1&\phantom{-}1&1\\-1&-1&-1&-1&1&1&1&1
\end{pmatrix*}.\]
Furthermore, by $\Lambda_n'= \big( \bigcup_{x \in \Lambda_n} (x + \eps_n \{ z^1, \ldots, z^8 \}) \big) \cap \big( \R^2 \times (0, h_n) \big)$ we denote the set of midpoints of lattice cells $x + [-\eps_n /2,\eps_n /2]^3$ contained in $\R^2 \times [0, h_n]$ for which at least one corner lies in $\Lambda_n$.  Additionally, let $\vec{w}(x) = \frac{1}{\eps_n}(w(x+\eps_n z^1), \dots, w(x+\eps_n z^8)) \in \R^{3 \times 8}$. Then, we assume that the atomic interaction energy for a deformation map $w \colon \Lambda_n \to \R^3$ can be written as
\begin{equation} \label{energyinteratom}
E_{\rm atom}(w) = \sum_{x \in \Lambda_n'} W(x,\vec{w}(x)),
\end{equation}
where $W(x,\cdot) : \R^{3\times 8} \to [0,\infty)$ only depends on those $\vec{w}_i$ with $x+\eps_n z^i \in \Lambda_n$, which makes \eqref{energyinteratom} meaningful even though $w$ is only defined on $\Lambda_n$.

As a full interaction model with long-range interaction would be significantly more complicated in terms of notation and would result in a much more complicated limit for finitely many layers, we restrict ourselves to these cell energies.

In the following we will sometimes discuss the upper and lower part of a cell separately. We write $A = (A^{(1)}, A^{(2)})$ with $A^{(1)}, A^{(2)} \in \R^{3 \times 4}$ for a $3 \times 8$ matrix $A$. 

If the full cell is occupied by atoms, i.e., $x+ \eps_n z^i \in \Lambda_n$ for all $i$, then we assume that $W$ is is given by a homogeneous cell energy $W_{\rm cell}:\R^{3\times 8} \to [0,\infty)$ with the addition of a homogeneous surface energy $W_{\rm surf}:\R^{3\times 4} \to [0,\infty)$ at the top and bottom. That means,
\[W(x,\vec{w}) = \begin{cases} 
        W_{\rm cell}(\vec{w}) & \text{if } x_3 \in (\eps_n/2, h_n - \eps_n/2), \\ 
        W_{\rm cell}(\vec{w})+ W_{\rm surf}(\vec{w}^{(2)}) & \text{if } \nu_n \geq 3  \text{ and } x_3 = h_n-\eps_n/2, \\ 
        W_{\rm cell}(\vec{w}) + W_{\rm surf}(\vec{w}^{(1)}) & \text{if } \nu_n \geq 3 \text{ and } x_3 = \eps_n/2, \\   
        W_{\rm cell}(\vec{w}) + \sum_{i=1}^2 W_{\rm surf}( \vec{w}^{(i)}) & \text{if }  \nu_n = 2,  \text{ and } x_3 = h_n/2.
     \end{cases}\]

\begin{example}\label{ex:mass-spring}
A basic example is given by a mass-spring model with nearest and next to nearest neighbor interaction: 
\begin{align*} 
  E_{\rm atom}(w) 
  &= \frac{\alpha}{4} \sum_{x,x' \in \Lambda_n \atop |x - x'| = \eps_n} \Big( \frac{|w(x) - w(x')|}{\eps_n} - 1 \Big)^2  + \frac{\beta}{4} \sum_{x,x' \in \Lambda_n \atop |x - x'| = \sqrt{2} \eps_n} \Big( \frac{|w(x) - w(x')|}{\eps_n} - \sqrt{2} \Big)^2. 
\end{align*} 
$E_{\rm atom}$ is of the form \eqref{energyinteratom} if we set 
\begin{align*} 
  W_{\rm cell}(\vec{w})
  &= \frac{\alpha}{16} \sum_{1 \le i,j \le 8 \atop |z^i - z^j| = 1} \big( |w_i - w_j| - 1 \big)^2 
  + \frac{\beta}{8} \sum_{1 \le i,j \le 8 \atop |z^i - z^j| = \sqrt{2}} \big( |w_i - w_j| - \sqrt{2} \big)^2  
\end{align*} 
and 
\begin{align*} 
  W_{\rm surf}(w_1,w_2,w_3,w_4)
  &= \frac{\alpha}{8} \sum_{1 \le i,j \le 4 \atop |z^i - z^j| = 1} \big( |w_i - w_j| - 1 \big)^2 \\ 
  &\qquad\qquad + \frac{\beta}{8} \sum_{1 \le i,j \le 4 \atop |z^i - z^j| = \sqrt{2}} \big( |w_i - w_j| - \sqrt{2} \big)^2.  
\end{align*}
\end{example}

We will also allow for energy contributions from body forces $f_n \colon \Lambda_n \to \R^3$ given by
\[E_{\rm body}(w)=\sum_{x \in \Lambda_n} w(x) \cdot f_n(x).\]
We will assume that the $f_n$ do not depend on $x_3$, that $f_n(x)=0$ for $x$ in an atomistic neighborhood of the lateral boundary, see \eqref{eq:force-bdry0}, and that there is no net force or first moment,
\begin{align}\label{eq:force-bed}
  \sum_{x \in \Lambda_n} f_n(x) =0, \quad \sum_{x \in \Lambda_n} f_n(x) \otimes x'=0, 
\end{align}
to not give a preference to any specific rigid motion. At last, we assume that after extension to functions $\bar{f}_n$ which are piecewise constant on each $x + (-\frac{\eps_n}{2}, \frac{\eps_n}{2})^2$, $x \in \eps_n \Z^2$,  $h_n^{-3} \bar{f}_n \to f$ in $L^2(S)$.

Overall, the energy is given as the sum
\begin{equation}\label{eq:energy}
E_n(w) = \frac{\eps_n^3}{h_n} \big(E_{\rm atom}(w) + E_{\rm body}(w) \big).
\end{equation}
Due to the factor $\frac{\eps_n^3}{h_n}$ this behaves like an energy per unit (undeformed) surface area.

Let us make some additional assumptions on the interaction energy. We assume that $W_{\rm cell}$, $W_{\rm surf}$, and all $W(x, \cdot)$ are invariant under translations and rotations, i.e., they satisfy 
$$ W(A) 
   = W(A + (c,\dots,c)) \mbox{ and } 
   W(RA)
   = W(A) $$ 
for any $A \in \R^{3 \times 8}$ or $A \in \R^{3 \times 4}$, respectively, and any $c \in \R^3$ and $R \in \SO(3)$. Furthermore, we assume that $W_{\rm cell}(Z)=W(x, Z) =0$, which in particular implies $W_{\rm surf}(Z^{(1)})=W_{\rm surf}(Z^{(2)})=0$, where $(Z^{(1)}, Z^{(2)}) = Z$. At last we assume that $W$ and $W_{\rm cell}$ are $C^2$ in a neighborhood of $Z$, while $W_{\rm surf}$ is $C^2$ in neighborhood of $Z^{(1)}$.

Since our model is translationally invariant, it is then equivalent to consider the discrete gradient
\[ \bar{\nabla} w(x) = \frac{1}{\eps_n} \big(w(x+\eps_n z^1) - \langle w \rangle , \dots, w(x+\eps_n z^8) - \langle w \rangle\big) \]
with
\[\langle w \rangle = \frac{1}{8} \sum_{i=1}^8 w(x+\eps_n z^i) \]
instead of $\vec{w}(x)$ for any $x$ with $x+ \eps_n z^i \in \Lambda_n$ for all $i$. In particular, the discrete gradient satisfies
\[ \sum_{i=1}^8 (\bar{\nabla} w(x))_{\cdot i} =0.\]
The bulk term is also assumed to satisfy the following single well growth condition.
\begin{itemize}
\item[{\bf (G)}] Assume that there is a $c_0>0$ such that
\[W_{\rm cell}(A) \geq c_0 \dist^2(A, \SO(3)Z)\]
for all $A \in \R^{3 \times 8}$ with $\sum_{i=1}^8 A_{\cdot i} =0$.
\end{itemize}


\subsection{Rescaling and Convergence of Displacements}
It turns out to be convenient to rescale our reference sets to the fixed domain $\Omega = S \times (0,1)$. For $x \in \R^3$ let us always write $x=(x', x_3)^T$ with $x'\in \R^2$. We define $\tilde{\Lambda}_n = H_n^{-1} \Lambda_n$ and $\tilde{\Lambda}_n'= H_n^{-1} \Lambda_n'$ with the rescaling matrix
\[H_n = \begin{pmatrix}
1 & 0 & 0 \\ 0 & 1 & 0 \\ 0& 0& h_n\\
\end{pmatrix}.\]

A deformation $w \colon \Lambda_n \to \R^3$ can be identified with the rescaled deformation $y \colon \tilde{\Lambda}_n \to \R^3$ given by $y(x) = w(H_n x)$. We then write $E_n(y)$ for $E_n(w)$. The rescaled discrete gradient is then given by
\[(\bar{\nabla}_n y(x))_{\cdot i} := \frac{1}{\eps_n}(y(x' + \eps_n (z^i)', x_3 + \frac{\eps_n}{h_n} z^i_3)-\langle y \rangle) = \bar{\nabla}w (H_n x)\]
for $x \in \tilde{\Lambda}_n'$, where now
\[\langle y \rangle = \frac{1}{8} \sum_{i=1}^8 y(x'+\eps_n (z^i)',x_3 + \frac{\eps_n}{h_n} z^i_3 ).\]

For a differentiable $v \colon \Omega \to \R^k$ we analogously set $\nabla_n v := \nabla v H_n^{-1} = (\nabla'v, \frac{1}{h_n} \partial_3 v)$.

In Section \ref{sec:preparations} we will discuss a suitable interpolation scheme with additional modifications at $\partial S$ to arrive at a $\dtilde{y}_n \in W^{1,2}(\Omega; \R^3)$ corresponding to $y_n$. Furthermore, for sequences in the von-Kármán energy scaling we will expect $y_n$ and $\dtilde{y}_n$ to be close to a rigid motion $x \mapsto R^*_n(x+c_n)$ for some $R^*_n, c_n$ and will therefore be interested in the normalized deformation
\begin{align}\label{eq:yn-tilde-def}
  \tilde{y}_n := {R^*_n}^T\dtilde{y}_n - c_n, 
\end{align}
which would then be close to the identity. The von-Kármán displacements in the limit will then be found as the limit objects of
\begin{align}
u_n(x') &:= \frac{1}{h_n^2} \int_0^1 (\tilde{y}_n)'-x'\,dx_3, \text{ and} \label{eq:un-def} \\
v_n(x') &:= \frac{1}{h_n} \int_0^1 (\tilde{y}_n)_3\,dx_3. \label{eq:vn-def} 
\end{align}


\subsection{The $\Gamma$-convergence result}

To describe the limit energy, let $Q_{\rm cell}(A) = D^2W_{\rm cell}(Z)[A,A]$ for $A\in\R^{3 \times 8}$ and $Q_{\rm surf}(A) = D^2W_{\rm surf}(Z^{(1)})[A,A]$ for $A \in \R^{3 \times 4}$. By frame indifference, 
\begin{align}\label{eq:Q-invariance}
  Q_{\rm cell}(A Z + c \otimes (1, \ldots, 1)) 
  = Q_{\rm surf}(A Z^{(1)} + c \otimes (1,1,1,1)) 
  = 0 
\end{align}
for all $c \in \R^3$ and all skew symmetric $A \in \R^{3 \times 3}$.

We introduce a relaxed quadratic form on $\R^{3 \times 8}$ by 
\begin{align*} 
  Q_{\rm cell}^{\rm rel}(A) 
  &= \min_{b \in \R^3} Q_{\rm cell} \big( a_1-\tfrac{b}{2}, \ldots, a_4-\tfrac{b}{2}, a_5+\tfrac{b}{2}, \ldots, a_8+\tfrac{b}{2} \big) \\ 
  &= \min_{b \in \R^3} Q_{\rm cell}(A + (b \otimes e_3) Z) 
   = \min_{b \in \R^3} Q_{\rm cell}(A + {\rm sym} (b \otimes e_3) Z). 
\end{align*} 
By Assumption {\bf (G)} $Q_{\rm cell}$ is positive definite on $(\R^3 \otimes e_3) Z$. Therefore, for each $A \in \R^{3 \times 8}$ there exists a (unique) $b = b(A)$ such that 
\begin{align}\label{eq:bmin-Q3} 
  Q_{\rm cell}^{\rm rel}(A) 
  &= Q_{\rm cell}(A + (b(A) \otimes e_3) Z) 
   = Q_{\rm cell}(A + {\rm sym} (b(A) \otimes e_3) Z)  
\end{align} 
and the mapping $A \mapsto b(A)$ is linear. (If $((v_i \otimes e_3) Z)_{i=1,2,3}$ is a $Q_{\rm cell}$-orthonormal basis of $(\R^3 \otimes e_3) Z$, then $b(A) = - \sum_{i=1}^3 Q_{\rm cell}[(v_i \otimes e_3) Z, A]$, where $Q_{\rm cell}[\cdot, \cdot]$ denotes the symmetric bilinear form corresponding to the quadratic form $Q_{\rm cell}(\cdot)$.)

At last, let us write
\[ Q_2(A) 
   = Q_{\rm cell}^{\rm rel} \bigg( \begin{pmatrix} A & 0 \\ 0 & 0 \end{pmatrix} Z \bigg), \qquad 
   Q_{2,{\rm surf}}(A) 
   = Q_{\rm surf} \bigg( \begin{pmatrix} A & 0 \\ 0 & 0 \end{pmatrix} Z \bigg) \]
for any $A \in \R^{2 \times 2}$.

We are now in place to state our main theorem in its first version. 

\begin{theorem} \label{thm:Gammalimit}

{\rm a)} If $\nu_n \to \infty$, then $\frac{1}{h_n^4} E_n\stackrel{\Gamma}{\longrightarrow} E_{\rm vK}$ with
\begin{align*}
  E_{\rm vK}(u,v,R^*)
  &:= \int_S \tfrac{1}{2} Q_2( G_1(x')) + \tfrac{1}{24} Q_2(G_2(x'))  + f(x') \cdot v(x') R^* e_3 \, dx',
\end{align*}
where $G_1(x') = {\rm sym} \nabla' u(x') + \tfrac{1}{2} \nabla' v(x') \otimes \nabla' v(x')$ and $G_2(x') = - (\nabla')^2 v(x')$.

More precisely, for every sequence $y_n$ with bounded energy $\frac{1}{h_n^4} E_n(y_n) \leq C$, there exists a subsequence (not relabeled), a choice of $R^*_n \in \SO(3), c_n \in \R^3$, and maps $u\in W^{1,2}(S;\R^2)$, $v \in W^{2,2}(S)$ such that $(u_n, v_n)$ given by \eqref{eq:un-def}, \eqref{eq:vn-def} and \eqref{eq:yn-tilde-def} satisfy $u_n \wto u$ in $W^{1,2}_{\rm loc}(S;\R^2)$, $v_n \to v$ in $W^{1,2}_{\rm loc}(S)$, $R^*_n \to R^*$, and
\[\liminf_{n \to \infty} \frac{1}{h_n^4} E_n(y_n) \geq E_{\rm vK}(u,v, R^*).\]
On the other hand, this lower bound is sharp, as for every $u\in W^{1,2}(S;\R^2)$, $v \in W^{2,2}(S)$, and $R^* \in \SO(3)$ there is a sequence $y_n$ such that $u_n \wto u$ in $W^{1,2}_{\rm loc}(S;\R^2)$, $v_n \to v$ in $W^{1,2}_{\rm loc}(S)$ (where we can take $R^*_n = R^*$, $c_n=0$ without loss of generality) and
\[\lim_{n \to \infty} \frac{1}{h_n^4} E_n(y_n) = E_{\rm vK}(u,v,R^*).\]
\medskip

\noindent {\rm b)} If $\nu_n \equiv \nu \in \N$, then $\frac{1}{h_n^4}  E_n\stackrel{\Gamma}{\longrightarrow} E^{(\nu)}_{\rm vK}$, to be understood in exactly the same way as in a), where
\begin{align*}
  E^{(\nu)}_{\rm vK}(u,v, R^*)
  &= \int_S \tfrac{1}{2} Q_{\rm cell}^{\rm rel} 
       \bigg( \begin{pmatrix} \sym G_1(x') & 0 \\ 0 & 0 \end{pmatrix} Z 
       + \tfrac{1}{2(\nu-1)} G_3(x') \bigg) \\ 
  &~~\qquad + \tfrac{\nu(\nu-2)}{24(\nu-1)^2} Q_{\rm cell}^{\rm rel} 
       \bigg( \begin{pmatrix} G_2(x') & 0 \\ 0 & 0 \end{pmatrix} Z\bigg) \\ 
  &~~\qquad + \tfrac{1}{\nu-1} Q_{\rm surf} \bigg( \begin{pmatrix} \sym G_1(x') & 0 \\ 0 & 0 \end{pmatrix} Z^{(1)} 
          + \frac{\partial_{12} v(x')}{4(\nu-1)} M^{(1)}  \bigg) \\ 
  &~~\qquad + \tfrac{1}{4(\nu-1)} Q_{\rm surf} \bigg( \begin{pmatrix} G_2(x) & 0 \\ 0 & 0 \end{pmatrix} Z^{(1)}\bigg) \\ 
  &~~\qquad   + \tfrac{\nu}{\nu-1} f(x') \cdot v(x') R^* e_3 \, dx'. 
\end{align*}
Here,
\begin{align} 
  G_3(x') 
  &= \begin{pmatrix} G_2(x') & 0 \\ 0 & 0 \end{pmatrix} Z_- + \partial_{12}v(x')  M , \label{eq:G3def} \\ 
  M 
  &= (M^{(1)}, M^{(2)}) 
   = \tfrac{1}{2} e_3 \otimes (+1, -1, +1, -1, +1, -1, +1, -1), \label{eq:M-def} \\ 
  Z_- 
  &= (-Z^{(1)},Z^{(2)}) 
   = (-z^1,-z^2,-z^3,-z^4,+z^5,+z^6,+z^7,+z^8). \label{eq:Zminus-def}
\end{align}
\end{theorem}

In the following we use the notation $E_{\rm vK}(u,v)$, respectively, $E^{(\nu)}_{\rm vK}(u,v)$, for the functionals without the force term.

\begin{example}\label{ex:mass-spring-orient}
Theorem~\ref{thm:Gammalimit} applies to the interaction energy of Example~\ref{ex:mass-spring} if $W_{\rm cell}$ is augmented by an additional penalty term $+ \chi(\vec{w})$ which vanishes in a neighborhood of $\SO(3) Z$ but is $\ge c > 0$ in a neighborhood of ${\rm O}(3)Z \setminus \SO(3) Z$, so as to guarantee orientation preservation. 
\end{example}

\begin{remark}
\begin{enumerate}
\item The result in a) is precisely the functional one obtains by first applying the Cauchy-Born rule (in 3d) in order to pass from the discrete set-up to a continuum model and afterwards computing the (purely continuum) $\Gamma$-limit on the energy scale $h^5$ as $h\to0$ as in \cite{FJM:06}. Indeed, the Cauchy-Born rule associates the continuum energy density 
\[ W_{\rm CB}(A) = W_{\rm cell}(A Z) \] 
to the atomic interaction $W_{\rm cell}$, and so $Q_{\rm cell}(AZ) = D^2W_{\rm CB}(Z)[A,A] =: Q_{\rm CB}(A)$ for $A\in\R^{3 \times 3}$, in particular, 
\[ Q_2(A) 
   = \min_{b \in \R^3} Q_{\rm CB} \bigg( \begin{pmatrix} A & 0 \\ 0 & 0 \end{pmatrix} + b \otimes e_3  \bigg). \] 

\item In contrast, for finite $\nu$ non-affine lattice cell deformations of the form $A Z_- + a M$, $A \in \R^{3 \times 3}$, $a \in \R$ need to be taken into account. While $A Z_-$ is non-affine in the out-of-plane direction, $aM$ distorts a lattice unit cell in-plane in a non-affine way. 

\item  Suppose that in addition $W_{\rm cell}$ and $W_{\rm surf}$ satisfy the following antiplane symmetry condition: 
\begin{align*} 
  W_{\rm cell}(w_1, \ldots, w_8) 
  &= W_{\rm cell} (P w_5, \ldots, P w_8, P w_1, \ldots, P w_4), \\ 
  W_{\rm surf}(w_1, \ldots, w_4) 
  &= W_{\rm surf}(P w_1, \ldots, P w_4), 
\end{align*} 
where $P$ is the reflection $P(x',x_3) = (x',-x_3)$. This holds true, e.g., in mass-spring models such as in Example~\ref{ex:mass-spring}. As both terms in $G_3$ switch sign under this transformation, while the affine terms with $G_1$ and $G_2$ remain unchanged, one finds that the quadratic terms in $E^{(\nu)}_{\rm vK}$ decouple in this case and we have 
 \begin{align*}
  E^{(\nu)}_{\rm vK}(u,v)
  &= \int_S \tfrac{1}{2} Q_2 ( G_1(x') ) 
      + \tfrac{\nu(\nu-2)}{24(\nu-1)^2} Q_2 ( G_2(x') ) 
      + \tfrac{1}{8(\nu-1)^2} Q_{\rm cell}^{\rm rel} ( G_3(x') ) \\ 
  &~~\qquad + \tfrac{1}{\nu-1} Q_{2,{\rm surf}} ( G_1(x') ) 
     + \tfrac{(\partial_{12} v(x'))^2}{16(\nu-1)^3} Q_{\rm surf} ( M^{(1)} ) \\ 
  &~~\qquad + \tfrac{1}{4(\nu-1)} Q_{2,{\rm surf}} ( G_2(x) ) \, dx' \\ 
  &= E_{\rm vK}(u,v) + \int_S \tfrac{1}{\nu-1} 
     \Big[ Q_{2,{\rm surf}} ( G_1(x') ) + \tfrac{1}{4} Q_{2,{\rm surf}} ( G_2(x) )  \Big] \\ 
  &\qquad\qquad\qquad\qquad  + \tfrac{1}{8(\nu-1)^2} \Big[Q_{\rm cell}^{\rm rel} ( G_3(x') )  - \tfrac{1}{3} Q_2 ( G_2(x') ) \Big] \\ 
  &\qquad\qquad\qquad\qquad\qquad  + \tfrac{1}{16(\nu-1)^3} (\partial_{12} v(x'))^2 Q_{\rm surf} ( M^{(1)} ) \, dx'. 
\end{align*}

\item Standard arguments in the theory of $\Gamma$-convergence show that for a sequence $(y_n)$ of almost minimizers of $E_n$ the in-plane displacement $u_n$, the out-of-plane displacement $v_n$ and the overall rotation $R^*_n$ converge (up to subsequences) to a minimizer $(u, v, R^*)$ of $E_{\rm vK}$, respectively, $E_{\rm vK}^{(\nu)}$.

\item For the original sequence $y_n$ near the lateral boundary there can be lattice cells for which only a subset of their corners belong to $\Lambda_n$. As a consequence these deformation cannot be guaranteed to be rigid on such cells and the scaled in-plane and out-of-plane displacements may blow up. We thus chose to modify in an atomistic neighborhood of the lateral boundary so as to pass to the globally well behaved quantities $\tilde{y}_n$, see Section \ref{sec:preparations}. For the original sequence $y_n$, Theorem~\ref{thm:Gammalimit} implies a $\Gamma$-convergence result with respect to weak convergence in $W^{1,2}_{\rm loc}$.

\end{enumerate}
\end{remark}

\medskip

\subsection{The $\Gamma$-convergence result under weaker assumptions}

One physically unsatisfying aspect of Theorem \ref{thm:Gammalimit} is the strong growth assumption {\bf (G)} which is in line with the corresponding continuum results \cite{FJM:06}. The problem is actually two-fold. First, typical physical interaction potentials, like Lennard-Jones potentials, do not grow at infinity but converge to a constant with derivatives going to $0$. And second, {\bf (G)} also implies that $W_{\rm cell}(-Z)>W_{\rm cell}(Z)$. In particular, the atomistic interaction could not even be $\Oo(3)$-invariant.

Contrary to the continuum case, it is actually possible to remove these restrictions in our atomistic approach. Indeed, if one assumes $\nu_n^5 \eps_n^2 \to 0$ or equivalently $h_n^5/\eps_n^3 \to 0$, then the von-Kármán energy scaling implies that the cell energy at every single cell must be small. In terms of the number of atom layers $\nu$, this condition includes the case of fixed $\nu$, as well as the case $\nu_n \to \infty$ as long as this divergence is sufficiently slow, namely $\nu_n \ll \eps_n^{-2/5}$.

In this case, growth assumptions at infinity should no longer be relevant. In fact, we can replace {\bf (G)} by the following much weaker assumption with no growth at infinity and full $\Oo(3)$-invariance.
\begin{itemize}
\item[{\bf (NG)}] Assume that $W_{\rm cell}(A)= W_{\rm cell}(-A)$ and that there is some neighborhood $U$ of $\Oo(3)Z$ and a $c_0>0$ such that
\[W_{\rm cell}(A) \geq c_0 \dist^2(A, \Oo(3)Z)\]
for all $A \in U$ with $\sum_{i=1}^8 A_{\cdot i} =0$ and 
\[W_{\rm cell}(A) \geq c_0\]
for all $A \notin U$ with $\sum_{i=1}^8 A_{\cdot i} =0$.
\end{itemize}
One natural problem arising from this is that atoms that are further apart in the reference configuration can end up at the same position after deforming. In particular, due to the full $\Oo(3)$-symmetry, neighboring cells can be flipped into each other without any cost to the cell energies, which completely destroys any rigidity that one expects in this problem.

As a remedy, whenever we assume {\bf (NG)}, we will add a rather mild non-penetration term to the energy that can be thought of as a minimal term representing interactions between atoms that are further apart in the reference configuration. To make this precise, for small $\delta, \gamma > 0$ let $V \colon \R^3\times \R^3 \to [0,\infty]$ be any function with $V(v,w) \geq \gamma$ if $\lvert v -w \rvert < \delta$ and $V(v,w) = 0$ if $\lvert v -w \rvert \geq 2\delta$. Then define
\[E_{\rm nonpen}(w) =  \sum_{x,\bar{x} \in \Lambda_n} V\Big(\frac{w(x)}{\eps},\frac{w(\bar{x})}{\eps}\Big).\]
Then, $\gamma >0$ ensures that there is a positive energy contribution whenever two atoms are closer than $\delta \eps$.

The overall energy is then given by
\begin{equation} \label{eq:energywithnonpen}
E_n(w) = \frac{\eps_n^3}{h_n} \big(E_{\rm atom}(w) + E_{\rm body}(w) + E_{\rm nonpen}(w) \big).
\end{equation}

\begin{theorem} \label{thm:Gammalimit2}
Assume that $\nu_n^5 \eps_n^2 \to 0$, that $f_n=0$, that $E_n$ is given by \eqref{eq:energywithnonpen}, and that
{\bf (G)} is replaced by {\bf (NG)}. Then all the statements of Theorem \ref{thm:Gammalimit} remain true, where now $R^*_n, R^* \in \Oo(3)$.
\end{theorem}

Note that in this version, we assume $f_n =0$. Indeed, if one were to include forces, one can typically reduce the energy by moving an atom infinitely far away in a suitable direction. Without any growth assumption in the interaction energy this can easily lead to $\inf E_n = - \infty$ and a loss of compactness. However, this is just a problem about global energy minimization. Not only should there still be well-behaved local minima of the energy, but the energy barrier in between should become infinite in the von-Kármán energy scaling.

In the spirit of local $\Gamma$-convergence, we can thus consider the set of admissible functions
\[
\mathcal{S}_\delta = \{w: \Lambda_n \to \R^3 \text{ such that } \dist(\bar{\nabla} w(x),\SO(3)Z) <\delta \text{ for all } x \in {\Lambda'_n}^{\circ}\}, 
\]
where ${\Lambda'_n}^{\circ}$ labels `interior cells' away from the lateral boundary, cf.\ Section~\ref{sec:preparations}.
 This leads us to the total energy
\begin{equation} \label{eq:energylocal}
E_n(w) = \begin{cases} 
        \frac{\eps_n^3}{h_n} \big(E_{\rm atom}(w) + E_{\rm body}(w) \big) & \text{if } w \in \mathcal{S}_\delta, \\ 
        \infty & \text{else}.
     \end{cases}
\end{equation}

We then have a version of the $\Gamma$-limit that does allow for forces.
\begin{theorem}\label{thm:Gammalimit3}
Assume that $\nu_n^5 \eps_n^2 \to 0$, that $E_n$ is given by \eqref{eq:energylocal} with $\delta>0$ sufficiently small, and that {\bf (G)} is replaced by {\bf (NG)}. Then all the statements of Theorem \ref{thm:Gammalimit} remain true. Furthermore, there is an infinite energy barrier in the sense that
\[ \lim_{n \to \infty} \inf \Big\{ \frac{1}{h_n^4} E_n(w) : w \in \mathcal{S}_\delta \backslash \mathcal{S}_{\delta/2} \Big\}= \infty.\]
\end{theorem}

\begin{remark}
\begin{enumerate}
\item For $n$ large enough, the energy barrier implies that minimizers of the restricted energy \eqref{eq:energywithnonpen} correspond to local minimizers of the unrestricted energy \eqref{eq:energy}. The results thus implies convergence of local minimizers of \eqref{eq:energy} in $\mathcal{S}_\delta$.

\item To formulate it differently, if a sequence $(w_n)$ is not separated by a diverging (unrestricted) energy barrier from the reference state $\id$, i.e.\ each $w_n$ can be connected by a continuous path of deformations $(w_n^t)_{t \in [0,1]}$ with equibounded energy $E_{\rm atom}(w_n^t) + E_{\rm body}(w_n^t)$, then $w_n \in \mathcal{S}_\delta$ for large $n$. This implies convergence of minimizers of the unrestricted energy under the assumption that a diverging energy barrier cannot be overcome. 

\item As the energy only has to be prescribed in $\mathcal{S}_\delta$, Theorem~\ref{thm:Gammalimit3} also describes local minimizers of energy functionals which are invariant under particle relabeling for point configurations which after labeling with their nearest lattice site by $\{w(x) : x \in \Lambda_n\}$ belong to $\mathcal{S}_{\delta}$, where their energy can be written in the form \eqref{eq:energylocal}.
\end{enumerate}
\end{remark}

\begin{example}\label{ex:mass-spring-LJ}
In the setting of Theorems~\ref{thm:Gammalimit2} and \ref{thm:Gammalimit3}, Example~\ref{ex:mass-spring-orient} can be generalized to energies of the form 
\begin{align*} 
  E_{\rm atom}(w) 
  &= \frac{\alpha}{4} \sum_{x,x' \in \Lambda_n \atop |x - x'| = \eps_n} V_1\Big( \frac{|w(x) - w(x')|}{\eps_n} - 1 \Big) + \frac{\beta}{4} \sum_{x,x' \in \Lambda_n \atop |x - x'| = \sqrt{2} \eps_n} V_2 \Big( \frac{|w(x) - w(x')|}{\eps_n} - \sqrt{2} \Big),  
\end{align*} 
where $V_1, V_2$ are pair interaction potentials with $V_i(0) = 0$, $V_i$ $C^2$ in a neighborhood of $0$ and $V_i(r) \ge c_0 \min\{ r^2, 1 \}$ for some $c_0 > 0$. (This is satisfied, e.g., for the Lennard-Jones potential $r \mapsto (1+r)^{-12} - 2 (1+r)^{-6} + 1$.) Due to the non-penetration term in \eqref{eq:energywithnonpen} no additional penalty terms for orientation preservation are necessary. Most notably, it is not assumed that $V_i(r) \to \infty$ as $r \to \infty$.
\end{example}

\section{Preparations} \label{sec:preparations}

We first extend a lattice deformation slightly beyond $\Lambda_n$, thereby possibly modifying near the lateral boundary $\partial S \times [0,h_n]$ where lattice cells might not be completely contained in $\bar{\Omega}_n$. Then we interpolate so as to obtain continuum deformations to which the continuum theory set forth in \cite{fjm02fvK,FJM:06} applies. 

For $x \in \Lambda_n'$, with $\Lambda_n'$ as defined at the beginning of Section~\ref{section:Models-and-Results}, we set 
\[ Q_n(x) 
   = x + (-\tfrac{\eps_n}{2}, \tfrac{\eps_n}{2})^3. \]
and also write $Q_n(\xi) = Q_n(x)$ whenever $\xi \in Q_n(x)$.

\subsection{Modification and Extension} 

On a cell that has a corner outside of $\Lambda_n $ there is no analogue to {\bf (G)} (or {\bf (NG)}) and hence no control of $\vec{w}(x)$ in terms of $W(x,\vec{w}(x))$. For this reason we modify our discrete deformations $w : \Lambda_n \to \R^3$ near the lateral boundary of $\Omega_n$. 

Let $S_n = \{ x \in S : \dist(x, \partial S) > \sqrt{2} \eps_n \}$ and note that, for $\eps_n > 0$ sufficiently small, $S_n$ is connected with a Lipschitz boundary. (This follows from the fact that $\partial S$ can be parameterized with finitely many Lip\-schitz charts.) If  $x \in \Lambda_n'$ is such that $\overline{Q_n(x)} \cap (S_n \times \R) \ne \emptyset$, we call $Q_{n}(x)$ an {\em inner cell} and write $x \in {\Lambda_n'}^{\circ}$. The corners of these cells are the interior atom positions $\Lambda_n^{\circ} = {\Lambda_n'}^{\circ} + \eps_n \{z^1, \ldots, z^8\}$ and the part of the specimen made of such inner cells is denoted 
$$ \Omega^{\rm in}_{n} 
   = \bigg(\bigcup_{x \in {\Lambda_n'}^{\circ}} \overline{Q_{n}(x)} \bigg)^{\circ}. $$
 Recall the definition of $\Lambda_n'$ from Section~\ref{section:Models-and-Results} and set  
$$ \bar{\Lambda}_n 
   = \Lambda_n' + \{ z^1, \ldots, z^8 \}, 
   \qquad
   \Omega^{\rm out}_{n} 
   = \bigg(\bigcup_{x \in \Lambda_n'} \overline{Q_{n}(x)} \bigg)^{\circ}. $$
The (lateral) {\em boundary cells} $Q_{n}(x)$ are those for which 
$$ x \in \partial \Lambda_n' := \Lambda_n' \setminus {\Lambda_n'}^{\circ}. $$ 
 Later we will also use the rescaled versions of these sets which are denoted $\tilde{\Lambda}_n = H_n^{-1} \Lambda_n$, $\tilde{\bar{\Lambda}}_n = H_n^{-1} \bar{\Lambda}_n$, ${\tilde{\Lambda}_n}^{\circ} = H_n^{-1} \Lambda_n^{\circ}$, $\tilde{\Lambda}'_n = H_n^{-1} \Lambda_n'$, $(\tilde{\Lambda}_n')^{\circ} = {\Lambda_n'}^{\circ}$. The rescaled lattice cells are $\tilde{Q}_n(x) = H_n^{-1} Q_n(H_n x)$.

If $w : \Lambda_n \to \R^3$ is a lattice deformation, following \cite{schmidtlinelast} we define a modification and extension $w' : \bar{\Lambda}_n \to \R^3$ as follows. First we set $w'(x) = w(x)$ if $x \in \Lambda_n^{\circ}$. Now partition $\partial \Lambda_n'$ into the $8$ sublattices $\partial \Lambda_{n,i}' = \partial \Lambda_n' \cap \eps_n ( z^i + 2\Z^3 )$. We apply the following extension procedure consecutively for $i = 1, \ldots, 8$:  

For every cell $Q = Q_{n}(x)$ with $x \in \partial \Lambda_{n,i}'$ such that there exists a neighboring cell $Q' = Q_{n}(x')$, i.e. sharing a face with $Q$, on the corners of which $w'$ has been defined already, we extend $w'$ to all corners of $Q$ by choosing an extension $w'$ such that $\dist^2({\bar{\nabla} w(x)}, \SO(3)Z)$ is minimal. 

As a result of this procedure, $w'$ will be defined on every corner of each cell neighboring an inner cell. Now we repeat this procedure until $w'$ is extended to $\bar{\Lambda}_n$, i.e., to every corner of all inner and boundary cells. Since $S$ is assumed to have a Lipschitz boundary, the number of iterations needed to define $w'$ on all boundary cells is bounded independently of $\eps$. 

This modification scheme guarantees that the rigidity and displacements of boundary cells can be controlled in terms of the displacements, respectively, rigidity of inner cells, see \cite[Lemmas~3.2 and 3.4]{schmidtlinelast}\footnote{We apply these lemmas without a Dirichlet part of the boundary, i.e., $\partial \mathcal{L}'_{\eps}(\Omega)_* = \emptyset$ in the notation of \cite{schmidtlinelast}. Note also that there is a typo in the statement of these lemmas. The set $\mathcal{B}_{\eps}$ should read $\{ \bar{x} \in \mathcal{L}_{\eps}'(\Omega)^{\circ} \cup \partial \mathcal{L}_{\eps}'(\Omega)_* : \bar{x} \notin V_{\eps} \}$, which in our notation (and without Dirichlet part of the boundary) is s subset of ${\Lambda_n'}^{\circ}$.}: 
\begin{lemma}\label{lemma:bdry-control}
There exist constants $c, C > 0$ (independent of $n$) such that for any $w : \Lambda_n \to \R^3$ and $R^* \in \SO(3)$
$$ \sum_{x \in \partial \Lambda_n'} 
   | \bar{\nabla} w'(x) - R^* Z |^2 
   \le C \sum_{x \in {\Lambda_n'}^{\circ}} | \bar{\nabla} w'(x) - R^* Z |^2 $$
as well as
$$ \sum_{x \in \partial \Lambda_n'} 
   \dist^2(\bar{\nabla} w'(x), \SO(3)Z) 
   \le C \sum_{x \in {\Lambda_n'}^{\circ}} \dist^2(\bar{\nabla} w'(x), \SO(3)Z). $$
\end{lemma}

For the sake of notational simplicity, we will sometimes write $w$ instead of $w'$.

\subsection{Interpolation} 

Let $w : \bar{\Lambda}_n \to \R^3$ be a (modified and extended) lattice deformation. We introduce two different interpolations: $\tilde{w}$ and $\bar{w}$. $\tilde{w} \in W^{1,2}(\Omega^{\rm out}_n; \R^3)$ is obtained by a specific piecewise affine interpolation scheme as in \cite{Sch:06,schmidtlinelast} which in particular associates the exact average of atomic positions to the center and to the faces of lattice cells. This will allow for a direct application of the results in \cite{FJM:06} on continuum plates. By way of contrast, $\bar{w}$ is a piecewise constant interpolation on the lattice Voronoi cells of $\bar{\Lambda}_n$. The advantage of this interpolation will be that a discrete gradient of $w$ translates into a continuum finite difference operator acting on $\bar{w}$.

Let $x \in \Lambda_n'$. In order to define $\tilde{w}$ on the cube $\overline{Q(x)}$ we first set $\tilde{w}(x) = \frac{1}{8} \sum_{i = 1}^8 w(x + \eps_n z^i)$. Next, for the six centers $v^1, \ldots, v^6$ of the faces $F^1, \ldots, F^6$ of $[-\frac{1}{2}, \frac{1}{2}]^3$ we set $\tilde{w}(x + \eps_n v^i) = \frac{1}{4} \sum_j w(x + \eps_n z^j)$, where the sum runs over those $j$ such that $z^j$ is a corner of the face with center $v^i$. Finally, we interpolate linearly on each of the 24 simplices 
\[ \operatorname{co} ( x, x + \eps_n v^k, x + \eps_n z^i, x + \eps_n z^j ) \] 
with $|z^i - z^j| = 1$, $|z^i - v^k| = |z^j - v^k| = \tfrac{1}{\sqrt{2}}$, i.e., whose corners are given by the cube center and the center and two neighboring vertices of one face. Note that for this interpolation
\begin{align}
\tilde{w}(x) &= \dashint_{Q(x)} \tilde{w}(\xi) \, d\xi, \\
\tilde{w}(x + \eps_n v^k) &= \dashint_{x + \eps_n F^k} \tilde{w}(\zeta) \, d \zeta, \label{eq:interpolsurf}
\end{align}
for every face $x + \eps_n F^k$ of $Q(x)$. 

For the second interpolation we first let $V_n^{\rm out} := \big( \bigcup_{x \in \bar{\Lambda}_n} ( x + [-\tfrac{\eps_n}{2}, \tfrac{\eps_n}{2}]^3 ) \big)^{\circ}$ and then define $\bar{w} \in L^2(V_n^{\rm out}; \R^3)$ by $\bar{w}(\xi) = w(x)$ for all $\xi \in x + (-\tfrac{\eps_n}{2}, \tfrac{\eps_n}{2})^3$, $x \in \bar{\Lambda}_n$.
Note that 
\[ \bar{\nabla} \bar{w}(x) = \frac{1}{\eps_n} \big(\bar{w}(x+\eps_n z^1) - \langle \bar{w} \rangle , \dots, \bar{w}(x+\eps_n z^8) - \langle \bar{w} \rangle\big) \]
with $\langle \bar{w} \rangle = \frac{1}{8} \sum_{i=1}^8 \bar{w}(x+\eps_n z^i)$
defines a piecewise constant mapping on $\Omega^{\rm out}_{n}$ such that 
\[ \bar{\nabla} \bar{w}(\xi) 
   = \bar{\nabla} w(x) 
   \quad \text{whenever} \quad \xi \in Q_n(x),~ x \in \Lambda_n'. \]

It is not hard to see that the original function controls the interpolation and vice versa. 
\begin{lemma}\label{lemma:interpol-Vergleich}
There exist constants $c, C > 0$ such that for any (modified, extended and interpolated) lattice deformation $\tilde{w} : \Omega^{\rm out}_{n} \to \R^3$ and any cell $Q = Q_n(x)$, $x \in \Lambda_n'$, 
\begin{align*}
  c |\bar{\nabla} w(x)|^2 
  \le \eps_n^{-3} \int_Q |\nabla \tilde{w}(\xi)|^2 \, d\xi 
  \le C |\bar{\nabla} w(x)|^2. 
\end{align*}
\end{lemma}

\begin{proof}
After translation and rescaling we may without loss assume that $\eps_n = 1$ and $Q = (0,1)^3$, hence $x = (\tfrac{1}{2}, \tfrac{1}{2}, \tfrac{1}{2})^T$. The claim then is an immediate consequence of the fact that both 
$$ \tilde{w} \mapsto |\bar{\nabla} \tilde{w}(x)| 
   \quad \mbox{and} \quad 
   \tilde{w} \mapsto \| \nabla \tilde{w} \|_{L^2(Q; \R^{3 \times 3})} $$ 
are norms on the finite dimensional space of continuous mappings $\tilde{w}$ which are affine on each $\operatorname{co} ( x, x + v^k, x + z^i, \eps_n z^j )$ with $|z^i - z^j| = 1$, $|z^i - v^k| = |z^j - v^k| = \tfrac{1}{\sqrt{2}}$, and which have $\int_Q \tilde{w}(\xi) \, d\xi = 0$. 
\end{proof}

\begin{lemma}\label{lemma:interpol-rigidity} 
There exist constants $c, C > 0$ such that for any (modified, extended and interpolated) lattice deformation $\tilde{w} : \Omega^{\rm out}_{n} \to \R^3$ and any cell $Q = Q_n(x)$, $x \in \Lambda_n'$, 
\begin{equation*}
  c \dist^2(\bar{\nabla} w(x), \SO(3)Z) 
  \leq \eps_n^{-3} \int_Q \dist^2 (\nabla \tilde{w}(\xi), \SO(3)) \, d\xi \leq C \dist^2(\bar{\nabla} w(x), \SO(3)Z). 
\end{equation*}
\end{lemma}

This is in fact \cite[Lemma 3.6]{schmidtlinelast}. We include a simplified proof.

\begin{proof}
After translation and rescaling we may without loss assume that $\eps_n = 1$ and $Q = (0,1)^3$. The geometric rigidity result \cite[Theorem 3.1]{FJM:02} (indeed, an elementary version thereof) yields 
\begin{equation*}
  c \min_{R \in \SO(3)} \| \nabla \tilde{w} - R \|_{L^{2}(Q)}^2 
  \leq \int_Q \dist^2 (\nabla \tilde{w}(\xi), \SO(3)) \, d\xi 
  \leq C \min_{R \in \SO(3)} \| \nabla \tilde{w} - R \|_{L^{2}(Q)}^2. 
\end{equation*}
By definition also 
\begin{align*}
  \dist^2(\bar{\nabla} w(x), \SO(3)Z) 
  = \min_{R \in \SO(3)} | \bar{\nabla} w(x) - RZ |. 
\end{align*}
The claim then follows from applying Lemma \ref{lemma:interpol-Vergleich} to $\xi \mapsto \tilde{w}(\xi) - R\xi$ for each $R \in \SO(3)$. 
\end{proof}

For a sequence $w_n$ of (modified and extended) lattice deformations $w_n : \bar{\Lambda}_n \to \R^3$ with interpolations $\tilde{w}_n : \Omega^{\rm out}_{n} \to \R^3$ and $\bar{w}_n : V_n^{\rm out} \to \R^3$ we consider the rescaled deformations  $\dtilde{y}_n : \tilde{\Omega}^{\rm out}_n \to \R^3$ defined by 
$$ \dtilde{y}_n(x) := \tilde{w}_n(H_n x) 
   \quad\mbox{with}\quad 
   \tilde{\Omega}^{\rm out}_n := H_n^{-1} \Omega^{\rm out}_n $$ 
and $\dbar{y}_n : \tilde{V}^{\rm out}_n \to \R^3$ defined by 
$$ \dbar{y}_n(x) := \bar{w}_n(H_n x) 
   \quad\mbox{with}\quad 
   \tilde{V}^{\rm out}_n := H_n^{-1} V^{\rm out}_n. $$ 
(Later we will normalize by a rigid change of coordinates to obtain $\tilde{y}_n$ and $\bar{y}_n$.) Their rescaled (discrete) gradients are 
$$ \nabla_n \dtilde{y}_n(x) := \nabla \tilde{w}_n(H_n x) 
   \quad\mbox{and}\quad 
   \bar{\nabla}_n \dbar{y}_n(x) := \bar{\nabla} \bar{w}_n(H_n x) $$ 
for all $x \in\tilde{\Omega}^{\rm out}_{n}$. Finally, the force $f_n$ after extension to $\bar{\Lambda}_n$ is assumed to satisfy 
\begin{align}\label{eq:force-bdry0}
  f_n(x) = 0 
  \quad \text{for}\quad 
  x \in \bar{\Lambda}_n \setminus \Lambda_n^{\circ}
\end{align}
and its the piecewise constant interpolation is $\bar{f}_n : \tilde{V}^{\rm out}_n \to \R^3$.

\begin{remark}\label{rmk:convergence-notions}
Suppose $\nu_n = \nu$ constant. We note that for a  sequence of mappings $y_n : \Lambda_n \to \R^3$, if $\dtilde{y}_n \to y$ in $L^2(\Omega; \R^3)$ then $y$ is continuous in $x_3$ and affine in $x_3$ on the intervals $(\frac{i-1}{\nu-1}, \frac{i}{\nu-1})$, $i = 1, \ldots, \nu$. Similarly, if $\dbar{y}_n \to y^*$ in $L^2(S \times (\frac{-1}{2(\nu-1)}, \frac{2\nu -1}{2(\nu-1)}); \R^3)$, then $y^*$ is constant in $x_3$ on the intervals $(\frac{2i-1}{2(\nu-1)}, \frac{2i + 1}{2(\nu-1)})$, $i = 0, \ldots, \nu -1$. 

Suppose $y, y^* \in L^2(\Omega; \R^3)$ are piecewise affine, respectively, constant in $x_3$ as detailed above with $y^*(x',x_3) = y(x',\frac{i}{\nu-1})$ if $x_3 \in (\frac{2i-1}{2(\nu-1)}, \frac{2i + 1}{2(\nu-1)})$, $i = 0, \ldots, \nu -1$. It is not hard to see that the following are equivalent. 
\begin{itemize}
\item $\dtilde{y} \to y$ in $L^2(\Omega; \R^3)$. 
\item $\dbar{y} \to y^*$ in $L^2(S \times (\frac{-1}{2(\nu-1)}, \frac{2\nu -1}{2(\nu-1)}); \R^3)$.   
\item $\frac{\eps_n^3}{h_n} \sum_{x \in \tilde{\Lambda}_n} | y_n(x) - \dashint_{x + (-\frac{\eps_n}{2},\frac{\eps_n}{2})^2 \times (-\frac{\eps_n}{2h_n},\frac{\eps_n}{2h_n})} y^*(\xi) \, d\xi|^2 \to 0$. 
\end{itemize} 
The same is true in case $\nu_n \to \infty$ for $y = y^*$ if in the second statement $S \times (\frac{-1}{2(\nu-1)}, \frac{2\nu -1}{2(\nu-1)})$ is replaced by $\Omega$. 

In particular, limiting deformations do not depend on the interpolation scheme. 
\end{remark}

\section{Proofs}\label{section:Proofs}
\subsection{Compactness}

For the compactness we will heavily use the corresponding continuum rigidity theorem from \cite[Theorem 3]{fjm02fvK} and \cite[Theorem 6]{FJM:06}:

\begin{theorem} \label{thm:platerigidity}
Let $y \in W^{1,2}(\Omega;\R^3)$ and set $\mathcal{I}= \mathcal{I}(y)= \int_\Omega \dist^2(\nabla_n y, \SO(3))\,dx$. Then there exists maps $R:S \to \SO(3)$ and $\tilde{R} \in W^{1,2}(S; \R^{3 \times 3})$ with $\lvert \tilde{R} \rvert \leq C$, and a constant $R^* \in \SO(3)$ such that
\begin{align}
\lVert \nabla_n y - R \lVert^2_{L^2(\Omega)} &\leq C \mathcal{I},\\
\lVert R - \tilde{R} \lVert^2_{L^2(S)} &\leq C \mathcal{I},\\
\lVert \nabla\tilde{R} \lVert^2_{L^2(S)} &\leq \frac{C \mathcal{I}}{h_n^2},\\
\lVert \nabla_n y - R^*  \lVert^2_{L^2(\Omega)} &\leq \frac{C \mathcal{I}}{h_n^2},\\
\lVert R - R^* \lVert^2_{L^p(S)} &\leq \frac{C_p \mathcal{I}}{h_n^2},\ \forall p<\infty.
\end{align}
Crucially, none of the constants depend on $n$, $y$, or $\mathcal{I}$.
\end{theorem}

Furthermore, we will also use the continuum compactness result \cite[Lemmas~4 and 5]{fjm02fvK} and \cite[Lemma~1, Eq.~(96), and Lemma~2]{FJM:06} based on the previous rigidity result applied to some sequence $(\hat{y}_n)$. 

\begin{theorem} \label{thm:contcomp}
Let $\hat{y}_n \in W^{1,2}(\Omega;\R^3)$ with $\mathcal{I}(\hat{y}_n)\leq Ch_n^4$. Then there are $R^*_n\in \SO(3)$, $c_n \in \R^3$ as well as a $u \in W^{1,2}(S;\R^2)$ and a $v \in W^{2,2}(S)$ such that $y_n = {R^*_n}^T \hat{y}_n -c_n$ satisfies
\begin{align}
\lVert \nabla_n y_n - R_n \lVert^2_{L^2(\Omega)} &\leq Ch_n^4 \label{eq:ty-closeto-SO}\\
\lVert R_n - \tilde{R}_n \lVert^2_{L^2(S)} &\leq Ch_n^4\\
\lVert \nabla\tilde{R} \lVert^2_{L^2(S)} &\leq Ch_n^2\\
\lVert \nabla_n y_n - \Id  \lVert^2_{L^2(\Omega)} &\leq Ch_n^2\label{eq:grad-yn-close-to-Id}\\
\int_\Omega (\nabla_n y_n)_{12} - (\nabla_n y_n)_{21}\,dx &=0. \label{eq:ty-no-skew}
\end{align}
And, up to extracting subsequences, 
\begin{align}
\frac{1}{h_n^2}\int_0^1 y_n' - x' \,dx_3=:u_n &\wto u \text{ in } W^{1,2}(S;\R^2),\ i=1,2, \label{eq:un-conv}\\
\frac{1}{h_n}\int_0^1 (y_n)_3  \,dx_3=:v_n &\to v \text{ in } W^{1,2}(S;\R),\label{eq:vn-conv}
\end{align}
\begin{align}
\frac{\nabla_n y_n - \Id}{h_n} =:A_n &\to A = e_3 \otimes \nabla' v - \nabla' v \otimes e_3 \text{ in } L^2(\Omega; \R^{3 \times 3}),\label{eq:grad-yn-Id-A}\\
2\frac{\sym( R_n - \Id)}{h_n^2} &\to A^2 \text{ in } L^p(S; \R^{3 \times 3}),\ \forall p<\infty \label{eq:sym-Rn-Id}, 
\end{align}
\begin{align}\label{eq:cont-strain-convergence}
\frac{R_n^T \nabla_n y_n - \Id}{h_n^2} 
\wto G \text{ in } L^2(\Omega; \R^{3 \times 3}), 
\end{align}
where the upper left $2 \times 2$ submatrix $G''$ of $G$ is given by 
\begin{align}\label{eq:cont-strain-identify-i}
  G''(x) 
  &= G_1(x') + (x_3 - \tfrac{1}{2}) G_2(x'), 
\end{align}
with 
\begin{align}\label{eq:cont-strain-identify-ii}
  \sym G_1 
  &= \tfrac{1}{2} ( \nabla' u + (\nabla u)^T ) + \nabla 'v \otimes \nabla ' v, \quad 
  G_2 
  = - (\nabla')^2 v.
\end{align}
\end{theorem}

The following proposition allows us to apply these continuum results.
\begin{proposition}\label{prop:energy-estimates}
In the setting of Theorem \ref{thm:Gammalimit}, consider a sequence $w_n$ with
\begin{equation} \label{eq:seqbddenergy}
 E_n(w_n) \leq C h_n^4
\end{equation}
Then,
\begin{equation} \label{eq:bdonSOdist}
0 \leq \mathcal{I}(\dtilde{y}_n)= \int_\Omega \dist^2(\nabla_n \dtilde{y}_n, \SO(3))\,dx \leq C h_n^4.
\end{equation}
Here, $\dtilde{y}_n\in W^{1,2}(\Omega;\R^3)$ is the rescaled, modified, and interpolated version of $w_n$ according to Section~\ref{sec:preparations}.

In the setting of Theorem \ref{thm:Gammalimit3} the statement remains is true as well, while in the setting of Theorem \ref{thm:Gammalimit2} \eqref{eq:bdonSOdist} is still true but now $\dtilde{y}_n$ is the rescaled, modified, and interpolated version of either $w_n$ or $-w_n$ where the correct sign does depend on $w_n$.
\end{proposition}
\begin{proof}
Rescaling the $w_n$ and applying the modification and interpolation steps from Section \ref{sec:preparations}, we have sequences $\dtilde{y}_n \in W^{1,2}(\Omega;\R^3)$ and $\dbar{y}_n \in L^{2}(\Omega;\R^3)$. In particular, we can use Theorem \ref{thm:platerigidity} for this sequence.

 Take $R^*_n$ according to Theorem~\ref{thm:contcomp}. Then by Lemmas~\ref{lemma:bdry-control} and \ref{lemma:interpol-rigidity},
\[ \frac{\eps_n^3}{h_n}\sum_{ x \in \tilde{\Lambda}_n'} \lvert \bar{\nabla}_n \dbar{y}_n (x) - R^*_n Z \rvert^2 
   \leq C\int_\Omega \lvert \nabla_n \dtilde{y}(x) - R^*_n \rvert^2 \,dx \leq C \frac{\mathcal{I}_n}{h_n^2}.\]
A standard discrete Poincaré-inequality then shows
\[ \frac{\eps_n^3}{h_n}\sum_{ x \in {\tilde{\Lambda}_n}^{\circ}} \Big\lvert \dbar{y}_n(x) - R^*_n \begin{pmatrix}x' \\ hx_3 \end{pmatrix} -  \bar{c}_n \Big\rvert^2 
   \leq\frac{\eps_n^3}{h_n}\sum_{ x \in \tilde{\Lambda}_n'} \lvert \bar{\nabla}_n \dbar{y}_n (x) - R^*_n Z \rvert^2 \leq C \frac{\mathcal{I}_n}{h_n^2} \]
for a suitable $\bar{c}_n \in \R^3$. Now $f_n$ does not depend on $x_3$, vanishes close to $\partial S$ where the modification takes place, and satisfies $\sum_{x \in \Lambda_n} f_n =0$, as well as $\sum_{x \in \Lambda_n} f_n \otimes x' =0$. Hence, we see that
\begin{align*}
  \frac{\eps_n^3}{h_n} E_{\rm body}(w_n) 
  &= \frac{\eps_n^3}{h_n} \sum_{x \in {\tilde{\Lambda}_n}^{\circ}} f_n(x') \cdot y_n(x)\\
  &= \frac{\eps_n^3}{h_n} \sum_{x \in {\tilde{\Lambda}_n}^{\circ}} f_n(x') \cdot \Big(\dbar{y}_n(x) - R^*_n \begin{pmatrix}x' \\ hx_3 \end{pmatrix} - \bar{c}_n \Big).
\end{align*}
Using $\lVert  \bar{f}_n \rVert_{L^2(S)} \leq C h_n^3$ and abbreviating $\mathcal{I}(\dtilde{y}_n) = \mathcal{I}_n$, we thus find
\[\Big\lvert \frac{\eps_n^3}{h_n} E_{\rm body}(w_n) \Big\rvert \leq C \sqrt{\mathcal{I}_n} h_n^2.\]
On the other hand, due to ${\bf (G)}$  and Lemmas~\ref{lemma:bdry-control} and \ref{lemma:interpol-rigidity} we have
\begin{align*}
  \frac{\eps_n^3}{h_n} E_{\rm atom}(w_n) 
  &\geq c_0 \frac{\eps_n^3}{h_n} \sum_{ x \in ({\tilde{\Lambda}_n}')^{\circ}} \dist^2(\bar{\nabla}_n  y_n  (x),\SO(3)Z) \\
  &\geq c \frac{\eps_n^3}{h_n} \sum_{ x \in \tilde{\Lambda}'_n} \dist^2(\bar{\nabla}_n \dbar{y}_n (x),\SO(3)Z) 
   \geq c \mathcal{I}_n.
\end{align*}
Hence,
\begin{align*}
  0 
  \leq \mathcal{I}_n &\leq C \frac{\eps_n^3}{h_n} E_{\rm atom}(w_n)
  \leq C h_n^4 + C\frac{\eps_n^3}{h_n} \lvert E_{\rm body}(w_n) \rvert
   \leq C h_n^4 + C \sqrt{\mathcal{I}_n} h_n^2.
\end{align*}
We thus have
\begin{equation*}
0 \leq \mathcal{I}_n \leq C h_n^4.
\end{equation*}

All these statements remain true in the setting of Theorem \ref{thm:Gammalimit3} as the Assumptions ${\bf (G)}$ and ${\bf (NG)}$ are equivalent on $\mathcal{S}_\delta$. 

Now, consider the setting of Theorem \ref{thm:Gammalimit2} with Assumption ${\bf (NG)}$ instead of ${\bf (G)}$, as well as $f_n=0$ and $\nu_n^5 \eps_n^2 \to 0$ with the energy given by \eqref{eq:energywithnonpen}. Using \eqref{eq:seqbddenergy}, we find
\[0 \leq W_{\rm cell}(\bar{\nabla} w(x)) \leq C \frac{h_n^5}{\eps_n^3}\]
for every $x \in {\Lambda'_n}^{\circ}$ and
\[0 \leq V\Big(\frac{w_n(\bar{x})}{\eps_n},\frac{w_n(\dbar{x})}{\eps_n}\Big) \leq C \frac{h_n^5}{\eps_n^3}\]
for all $\bar{x}, \dbar{x} \in \Lambda_n$. As $\frac{h_n^5}{\eps_n^3} \to 0$, for $n$ large enough, the right hand side is strictly smaller then $c_0$ or $\gamma$, respectively. Therefore, for all $n$ large enough we have
\[\bar{\nabla} w_n(x) \in U  \text{ for all } x \in {\Lambda'_n}^{\circ}\]
and
\begin{equation} \label{eq:nonpen}
\lvert w_n(\bar{x}) - w_n( \dbar{x} ) \rvert > \eps_n \delta
\end{equation}
for all $\bar{x}, \dbar{x} \in \Lambda_n$.

$\bar{\nabla} w_n(x) \in U $ implies $W_{\rm cell}(\bar{\nabla} w_n(x)) \geq c_0 \dist^2(\bar{\nabla} w_n(x), \Oo(3)Z) $. In particular, we thus find
\[\dist^2(\bar{\nabla} w_n(x), \Oo(3)Z) \leq C \frac{h_n^5}{\eps_n^3}.\]
Again, for $n$ large enough, this means that every $x\in {\Lambda'_n}^{\circ}$ the discrete gradient $\bar{\nabla} w_n(x)$ is arbitrarily close to $\Oo(3)Z$ and thus very close to $\sigma_n(x)\SO(3)Z$ with a unique $\sigma_n(x) \in \{\pm 1\}$. We now want to show that the sign $\sigma_n(x)$ is the same for all $x$ in the interior cells. As the interior of the union of all these cells is connected, it suffices to show that $\sigma_n$ is the same on any two cells that share a $(d-1)$-face. Indeed, if that were false, we would have some $x,x'$ in cells that share a $(d-1)$-face such that \[\dist^2(\bar{\nabla} w_n(x), \Oo(3)Z)= \lvert \bar{\nabla} w_n(x) - QZ \rvert^2 \leq C \frac{h_n^5}{\eps_n^3},\]
and
\[\dist^2(\bar{\nabla} w_n(x'), \Oo(3)Z)= \lvert \bar{\nabla} w_n(x') + Q'Z \rvert^2 \leq C \frac{h_n^5}{\eps_n^3},\]
with $Q,Q' \in \SO(3)$. Without loss of generality assume $x=x' + \eps_n e_3$. Then\[\bar{\nabla} w_n(x') (0, b)^T = \bar{\nabla} w_n(x) (b, 0)^T\]
for all $b\in \R^4$ with $\sum_i b_i =0$. In particular choosing $b = (-1,+1,+1,-1)$ and $b = (-1,-1,+1,+1)$, we get $\lvert (Q+ Q') e_i \rvert \leq C \frac{h_n^5}{\eps_n^3}$ for $i=1,2$. As $Q,Q' \in \SO(3)$, we find $\lvert (Q- Q') e_3 \rvert \leq C \frac{h_n^5}{\eps_n^3}$. Overall, we see that both deformed cells are almost on top of each other. More specifically,
\begin{align*}
&\lvert w(x'+\eps_n z^1) - w(x + \eps_n z^5) \rvert\\
&~~= \lvert w(x+\eps_n z^5) - w(x+\eps_n z^1) + w(x'+\eps_n z^5) - w(x'+\eps_n z^1) \rvert \\
&~~\leq \eps_n \Big( \lvert Q z^5 - Qz^1 - Q' z^1 + Q'z^5 \rvert +C \frac{h_n^5}{\eps_n^3} \Big) \\
&~~= \eps_n \Big( \lvert (Q  - Q') e_3 \rvert +C \frac{h_n^5}{\eps_n^3} \Big) 
  \leq \eps_n C \frac{h_n^5}{\eps_n^3} 
  \leq \delta \eps_n
\end{align*}
for $n$ large enough. This is a contradiction to the non-penetration condition \eqref{eq:nonpen}.

That means, we have
\begin{equation}
\frac{\eps_n^3}{h_n}\sum_{ x \in {\Lambda'_n}^{\circ}} \dist^2(\sigma_n \bar{\nabla} w_n(x), \SO(3)Z) \leq C h_n^4 \nonumber
\end{equation}
for an $x$-independent $\sigma_n \in \{\pm 1\}$.
Applying the modification and interpolation procedure from Section~\ref{sec:preparations} to $\sigma_n w_n$ as in the case {\bf (G)} above, we find
\[
\int_\Omega \dist^2( \nabla_n \dtilde{y}(x), \SO(3)Z)\,dx \leq C h_n^4. \qedhere \]
\end{proof}

Now we can directly apply Theorems \ref{thm:platerigidity} and \ref{thm:contcomp} for the continuum objects $\dtilde{y}_n$. In particular, for $\tilde{y}_n = {R^*_n}^T \dtilde{y}_n - c_n$ as defined in \eqref{eq:yn-tilde-def} and corresponding $u_n$ and $v_n$ as in \eqref{eq:un-def}, respectively, \eqref{eq:vn-def}, after extracting a subsequence from \eqref{eq:un-conv} and \eqref{eq:vn-conv} we get that 
\begin{align}\label{eq:unvn-conv} 
  u_n \wto u \text{ in } W^{1,2}(S;\R^2), \qquad 
  v_n \to v \text{ in } W^{1,2}(S;\R).
\end{align} 
For later we also introduce $\bar{y}_n = {R^*_n}^T \dbar{y}_n - c_n$. 

We will also use the following finer statement.
\begin{proposition}\label{prop:finer-statement}
In the setting of Theorem \ref{thm:contcomp}, applied to $\dtilde{y}_n$ and with $\tilde{y}_n = {R^*_n}^T \dtilde{y}_n - c_n$, we have
\begin{align}
\frac{1}{h_n^2} \big((\tilde{y}_n)' - x'\big)=:\hat{u}_n &\wto \hat{u} \text{ in } W^{1,2}(\Omega;\R^2), \label{eq:hat-u-conv} \\
\frac{1}{h_n} (\tilde{y}_n)_3=:\hat{v}_n &\wto \hat{v} \text{ in } W^{1,2}(\Omega), \label{eq:hat-v-conv} 
\end{align}
where
\begin{align}
\hat{u}(x)&=u(x') - (x_3-\tfrac{1}{2}) \nabla' v(x'),  \label{eq:hat-u-form}\\
\hat{v}(x)&=v(x') + (x_3-\tfrac{1}{2}). \label{eq:hat-v-form}
\end{align}
\end{proposition}
\begin{proof}
According to Korn's inequality
\begin{align*}
\lVert \hat{u}_n \rVert_{W^{1,2}(\Omega; \R^2)} &\leq C \Big( \lVert \sym  \nabla' \hat{u}_n \rVert_{L^2(\Omega; \R^{2\times 2})}+ \Big\lVert \frac{\partial\hat{u}_n}{\partial x_3} \Big\rVert_{L^2(\Omega; \R^2)}\\
&\quad + \Big\lvert \int_{\Omega}  \skewo \nabla' \hat{u}_n \,dx \Big\rvert + \Big\lvert \int_{\Omega} \hat{u}_n \,dx \Big\rvert \Big).
\end{align*}
According to Theorem \ref{thm:contcomp}, $\sym  \nabla' \hat{u}_n$ is bounded in $L^2$ by \eqref{eq:ty-closeto-SO} and \eqref{eq:sym-Rn-Id}, $\int \skewo  \nabla' \hat{u}_n \,dx=0$  by \eqref{eq:ty-no-skew}, and $\int \hat{u}_n \,dx$ is bounded due to \eqref{eq:un-conv}. As
\begin{equation*}
\frac{\partial(\hat{u}_n)_i}{\partial x_3}  = \frac{1}{h_n} ( \nabla_n \tilde{y}_n - \Id )_{i3},
\end{equation*}
$i=1,2$, this term is bounded in $L^2$ as well. This shows compactness. To identify the limit and thus show convergence of the entire sequence, note that
\begin{equation*}
\int_0^1 \hat{u}_n \,dx_3 \wto u \text{ in } W^{1,2}(S;\R^2),
\end{equation*}
by \eqref{eq:un-conv} and
\begin{equation*}
\frac{\partial(\hat{u}_n)_i}{\partial x_3}  = \frac{1}{h_n} ( \nabla_n \tilde{y}_n - \Id )_{i3} \to - \frac{\partial v}{\partial x_i} \text{ in } L^2(\Omega),
\end{equation*}
for $i=1,2$ by \eqref{eq:grad-yn-Id-A}.

\eqref{eq:grad-yn-close-to-Id} and \eqref{eq:vn-conv} in Theorem \ref{thm:contcomp} also show that $\hat{v}_n$ is bounded in $W^{1,2}(\Omega)$ with $\frac{\partial  \hat{v}_n}{\partial x_3} \to 1$ and
\[ \int_0^1 \hat{v}_n \,dx_3 \to v. \qedhere \]
\end{proof}

As a first consequence, we will now describe the limiting behavior of the force term $E_{\rm body}(w_n) = E_{\rm body}(y_n) = \sum_{x \in \tilde{\bar{\Lambda}}_n} f_n(x) \cdot y_n(x)$, where $f_n(x) = f_n(x')$ satisfies \eqref{eq:force-bed}, \eqref{eq:force-bdry0} and $h_n^{-3} \bar{f}_n \to f$ in $L^2(S)$. 

Note that the forces considered are a bit more general than in \cite{FJM:06}.

\begin{proposition}\label{prop:limiting-force} Let $y_n$ be a sequence with $E_n(y_n) \le C h_n^4$ and suppose that \eqref{eq:unvn-conv} holds true for $\tilde{y}_n$, $u_n$, $v_n$ as defined in \eqref{eq:yn-tilde-def}, \eqref{eq:un-def}, \eqref{eq:vn-def}. Assume that $R^*_n \to R^*$. Then 
\[ \frac{\eps_n^3}{h_n^5} E_{\rm body}(y_n) 
   \to \begin{cases} 
          \int_S f(x') \cdot v(x') R^*e_3 \, dx', & \text{if } \nu_n \to \infty, \\ 
          \frac{\nu}{\nu-1} \int_S f(x') \cdot v(x') R^*e_3 \, dx', & \text{if } {\nu_n = \nu \text{ constant}}, \\ 
       \end{cases} \] 
as $n \to \infty$. 
\end{proposition}

\begin{proof}
In terms of the extended and interpolated force density we have  
\begin{align*}
  \frac{\eps_n^3}{h_n^5} E_{\rm body}(y_n) 
  &= \frac{1}{h_n^4} \int_{\tilde{V}^{\rm out}_n} \bar{f}_n(x) \cdot \dbar{y}_n(x) \, dx \\ 
  &= \frac{1}{h_n^4} \int_{\tilde{V}^{\rm out}_n} \bar{f}_n(x) \cdot \Big( \dbar{y}_n(x) - R^*_n \begin{pmatrix} x' \\ 0 \end{pmatrix} - R^*_n c_n \Big) \, dx \\ 
  &= \int_{\tilde{V}^{\rm out}_n} h_n^{-3} {R^*_n}^T \bar{f}_n(x) \cdot h_n^{-1} \Big( \bar{y}_n - \begin{pmatrix} x' \\ 0 \end{pmatrix} \Big) \, dx. 
\end{align*}
By Proposition~\ref{prop:finer-statement}, 
$ h_n^{-1} \Big( \tilde{y}_n - \begin{pmatrix} x' \\ 0 \end{pmatrix} \Big) 
   \to \hat{v} e_3 $ in $L^2(\Omega; \R^3)$
with $\hat{v}$ as in \eqref{eq:hat-v-conv} and so Remark~\ref{rmk:convergence-notions} shows that 
\begin{align*}
  \frac{\eps_n^3}{h_n^5} E_{\rm body}(y_n) 
  \to \int_{\Omega} {R^*}^T f(x) \cdot \hat{v}(x) e_3 \, dx 
  = \int_{\Omega} f(x') \cdot v(x') R^* e_3 \, dx'
\end{align*}
if $\nu_n \to \infty$, where in the last step we have used that \eqref{eq:force-bed} together with $f_n(x) = f_n(x')$ also implies that $\sum_{x \in \Lambda_n} x_3 f_n(x) = 0$. If $\nu_n = \nu$ constant, then Remark~\ref{rmk:convergence-notions} gives 
\begin{align*}
  \frac{\eps_n^3}{h_n^5} E_{\rm body}(y_n) 
  &\to \frac{1}{\nu-1} \sum_{j=0}^{\nu -1} \int_{S} {R^*}^T f(x') \cdot \hat{v}(x',\tfrac{j}{\nu-1}) e_3 \, dx' \\ 
  &= \frac{\nu}{\nu-1}  \int_{S}  f(x') \cdot v(x') R^* e_3 \, dx' 
\end{align*}
with an analogous argument for the last step. 
\end{proof}

\subsection{Lower bounds}

To show the lower bounds in our $\Gamma$-convergence results, we have to understand the limit of the discrete strain. Let $(y_n)$ satisfy $E_n(y_n) \leq C h_n^4$ and set 
\begin{equation*}
\bar{G}_n := \frac{1}{h_n^2} (R_n^T \bar{\nabla}_n \bar{y}_n - Z).
\end{equation*}
By Proposition~\ref{prop:energy-estimates} $(\dtilde{y}_n)$ satisfies the assumptions of Theorem~\ref{thm:contcomp} and Proposition~\ref{prop:finer-statement} so that, after a rigid change of coordinates, $\tilde{y}_n$ satisfies \eqref{eq:ty-closeto-SO}--\eqref{eq:cont-strain-identify-ii} and \eqref{eq:hat-u-conv}--\eqref{eq:hat-v-form}. In particular, by \eqref{eq:cont-strain-convergence} we know that for a subsequence the continuum strain converges as 
\[ \frac{1}{h_n^2} (R_n^T \nabla_n \tilde{y}_n - \Id) 
   \wto G  \text{ in }  L^2(\Omega; \R^{3 \times 3}), \] 
where $G$ satisfies \eqref{eq:cont-strain-identify-i} and \eqref{eq:cont-strain-identify-ii}.

For the discussion of discrete strains, recall that we defined
\begin{align*}
  Z_- 
  &= (-z^1, -z^2, -z^3 , -z^4, +z^5, +z^6, +z^7, +z^8), \\ 
  M 
  &= \frac{1}{2} e_3 \otimes (+1, -1, +1, -1, +1, -1, +1, -1). 
\end{align*} 
We define a projection $P$ acting on maps via 
\[ Pf(x) 
   = \dashint_{(k-1)/(\nu-1)}^{k/(\nu-1)} f(x',t) \, dt 
   \qquad \text{if} \qquad 
   \tfrac{k-1}{\nu-1} \le x_3 < \tfrac{k}{\nu-1} \] 
in case $\nu_n \equiv \nu < \infty$ and $P = \id$ in case $\nu_n \to \infty$.

\begin{proposition}\label{prop:limiting-strain}
Let $(y_n)_n$ satisfy $E_n(y_n) \leq C h_n^4$ with $\frac{1}{h_n^2} (R_n^T \nabla_n \tilde{y}_n - \Id) \wto G$ in $L^2(\Omega; \R^{3 \times 3})$. Then,
\[ \bar{G}_n 
   \wto \bar{G} 
   := \begin{cases} 
         GZ, & \text{if } \nu_n \to \infty, \\ 
         PGZ + \frac{1}{2(\nu-1)} G_3, & \text{if } \nu_n \equiv \nu \in \N,
      \end{cases} \] 
in $L^2(\Omega; \R^{3 \times 8})$, where $G_3$ is as in Theorem \ref{thm:Gammalimit}.
\end{proposition}
\begin{proof}
The compactness follows from Theorem \ref{thm:contcomp}. On a subsequence (not relabeled) we thus find $\bar{G}_n \wto \bar{G}$. As $R_n \to \Id$ in $L^2$ while being uniformly bounded, we also find
\begin{equation*}
R_n\bar{G}_n = \frac{1}{h_n^2} (\bar{\nabla}_n \bar{y}_n - R_n Z)  \wto \bar{G}.
\end{equation*}
We have
\[ \lim_{n \to \infty} \frac{1}{h_n^2} (R_n^T \nabla_n \tilde{y}_n - \Id) 
   = \lim_{n \to \infty} \frac{1}{h_n^2} (\nabla_n \tilde{y}_n - R_n)  
   = G, \] 
weakly in $L^2(\Omega; \R^{3 \times 3})$ where $G$ satisfies \eqref{eq:cont-strain-identify-i} and \eqref{eq:cont-strain-identify-ii}. 

In order to discuss the discrete strains in more detail, we separate affine and non-affine contributions. We say that a $b \in \R^8$ is affine if it is an element of the linear span of $b^0, b^1, b^2, b^3$, where $b^0 = (1, \ldots, 1)$ and $b^i = Z^T e_i$, $i = 1,2,3$. Any $b \in \R^8$ which is perpendicular to all affine vectors is called non-affine. I.e., a non-affine $b$ is characterized by $\sum_{i=1}^8 b_i =0$ and $Zb=0$. 

We begin by identifying the easier to handle affine part of the limiting strain. By construction we have $R_n \bar{G}_n b^0 \equiv 0$ and so $\bar{G} b^0 = 0 = G Z b^0$. For $i \in \{1, 2, 3\}$ we use that on any $\tilde{Q}_n(x)$, $x \in \tilde{\Lambda}'_n$,
\[ \bar{\nabla}_n \bar{y}_n (x) b^1 
   = \frac{1}{2\eps_n} \big( (y_2 + y_3 + y_6 + y_7) - ( y_1 + y_4 + y_5 + y_8) \big), \] 
where $y_i = \tilde{y}_n(x' + \eps_n (z^{i})', x_3 + \tfrac{\eps_n}{h_n} z^i_3)$. So, using \eqref{eq:interpolsurf} for $\tilde{y}_n$, 
\begin{align*}
  \bar{\nabla}_n \bar{y}_n (x) b^1 
  &= \frac{2}{\eps_n} \dashint_{x + \{-\frac{\eps_n}{2}\} \times (-\frac{\eps_n}{2}, \frac{\eps_n}{2}) \times (-\frac{\eps_n}{2h_n}, \frac{\eps_n}{2h_n})} \tilde{y}_n(\xi + \eps_n e_1) - \tilde{y}_n(\xi) \, d\xi \\ 
  &= 2 \dashint_{\tilde{Q}_n(x)} \partial_1 \tilde{y}_n(\xi) \, d\xi. 
\end{align*}
Analogous arguments yield 
\begin{align*}
  \bar{\nabla}_n \bar{y}_n (x) b^2
  &= 2 \dashint_{\tilde{Q}_n(x)} \partial_2 \tilde{y}_n(\xi) \, d\xi 
  \quad \mbox{and} \\
  \bar{\nabla}_n \bar{y}_n (x) b^3 
  &= \frac{2}{h_n} \dashint_{\tilde{Q}_n(x)} \partial_3 \tilde{y}_n(\xi) \, d\xi.
\end{align*}
By $P_n$ we denote the projection which maps functions to piecewise constant functions via $P_n f(x) = \dashint_{\tilde{Q}_n(x)} f(\xi) \, d\xi$ on $\tilde{Q}_n(x)$. Then $P_n [ R_n\bar{G}_n ] \wto \bar{G}$. On the other hand, observing that $Z Z^T = 2 \Id_{3\times 3}$, we find 
\begin{align*}
  P_n [ R_n\bar{G}_n ] b^i 
  = \frac{2}{h_n^2} P_n \big[ \partial_i \tilde{y}_n - R_n e_i \big] 
  \wto 2 P G e_i 
  = P G Z b^i, \quad i = 1,2 
\end{align*}
and
\begin{align*}
  P_n [ R_n\bar{G}_n ] b^3 
  = \frac{2}{h_n^2} P_n \big[ h_n^{-1} \partial_3 \tilde{y}_n - R_n e_3 \big] 
  \wto 2 P G e_3 
  = P G Z b^3. 
\end{align*}
In summary we get that for every affine $b \in \R^8$ 
\begin{align}\label{eq:Gbar-affine}
  \bar{G} b  
  = P G Z b. 
\end{align}

For the discussion of the non-affine part of the strain we fix a non-affine $b \in \R^8$, i.e., a $b$ satisfying $\sum_{i=1}^8 b_i =0$, $Zb=0$, and write  $b^T= ((b^{(1)})^T, (b^{(2)})^T)$, where $b^{(1)}, b^{(2)} \in \R^4$. Let $Z^{\rm 2dim} := ((z^1)', (z^2)', (z^3)', (z^4)') \in \R^{2 \times 4}$ be the matrix of two-dimensional directions. Then $Z^{\rm 2dim} (b^{(1)} + b^{(2)})=0$ and $\sum_{i=1}^4 b^{(1)}_i = \sum_{i=1}^4 b^{(2)}_i =0$. We introduce the difference operator 
\[ \bar{\nabla}^{\rm 2dim} f (x) 
   := \frac{1}{\eps_n} \Big( f(x' + \eps_n (z^i)', x_3) - \frac{1}{4} \sum_{j=1}^4 f(x' + \eps_n (z^j)', x_3) \Big)_{i=1,2,3,4}. \] 
The idea is now to separate differences into in-plane and out-of-plane differences, as all in-plane differences are infinitesimal, while out-of-plane differences stay non-trivial if $\nu_n \equiv \nu$ and have to be treated more carefully. 

Using 
\begin{align*}
  \bar{\nabla}_n \bar{y}_n(x) 
  &= \Big( \bar{\nabla}_n^{\rm 2dim} \bar{y}_n \big( x - \frac{\eps_n}{2h_n}e_3 \big), 
          \bar{\nabla}_n^{\rm 2dim} \bar{y}_n \big(x + \frac{\eps_n}{2h_n}e_3 \big) \Big) \\ 
  &\quad  + \frac{1}{2 h_n} \dashint_{\tilde{Q}_n(x)} \partial_3 \tilde y_n(\xi) \, d \xi \otimes (-1, -1, -1, -1, +1, +1, +1, +1) 
\end{align*}
we find
\begin{align}
  R_n \bar{G}_n(x) b
  &= \frac{1}{h_n^2} \bar{\nabla}_n \bar{y}_n(x) b\nonumber \\
  &=\frac{1}{h_n^2} \Big( \bar{\nabla}_n^{\rm 2dim} \bar{y}_n \big( x + \frac{\eps_n}{2h_n}e_3 \big) - \bar{\nabla}_n^{\rm 2dim} \bar{y}_n \big( x - \frac{\eps_n}{2h_n}e_3 \big)\Big) b^{(2)} \label{eq:strainidentification1}\\
  &\quad +\frac{1}{h_n^2} \bar{\nabla}_n^{\rm 2dim} \bar{y}_n \big( x - \frac{\eps_n}{2h_n}e_3 \big) (b^{(1)}+b^{(2)})\label{eq:strainidentification2},
\end{align} 
where we have used that $\sum_{i=1}^4 b^{(1)}_i = \sum_{i=1}^4 b^{(2)}_i = 0$. 

First consider the term \eqref{eq:strainidentification2}. Since $ \bar{\nabla}_n^{\rm 2dim} \bar{\id}(x - \frac{\eps_n}{2h_n}e_3)  = Z^{\rm 2dim}$  and $Z^{\rm 2dim} (b^{(1)}+b^{(2)}) = 0$, for any $\varphi \in C_c^\infty(\Omega)$ and $i=1,2$  by \eqref{eq:hat-u-conv} and Remark~\ref{rmk:convergence-notions} we have
\begin{align}
  &\frac{1}{h_n^2} e_i^T \int_\Omega \bar{\nabla}_n^{\rm 2dim} (\bar{y}_n - \bar{\id}) \big( x - \frac{\eps_n}{2h_n}e_3 \big) (b^{(1)}+b^{(2)}) \varphi(x)\,dx \notag \\ 
  &~~= \frac{1}{h_n^2} e_i^T \int_\Omega (\bar{y}_n - \bar{\id}) \big( x - \frac{\eps_n}{2h_n}e_3 \big) (\bar{\nabla}_n^{\rm 2dim})^* \varphi(x) (b^{(1)}+b^{(2)})\,dx \notag \\ 
  &~~\to -\int_\Omega \hat{u}_i(\tilde{x}) \nabla' \varphi(x) Z^{\rm 2dim} (b^{(1)}+b^{(2)})\,dx = 0,\label{eq:strainpart212}
\end{align}
where, either $\tilde{x}=x$ (if $\nu_n \to \infty$), or $\tilde{x} =  (x', \frac{\lfloor(\nu-1)x_3\rfloor}{\nu-1})$ (if $\nu_n =\nu$ is constant). 

For the third component, we instead have 
\begin{align*}
\frac{1}{h_n^2} e_3^T &\int_\Omega \bar{\nabla}_n^{\rm 2dim} \bar{y}_n \big( x - \frac{\eps_n}{2h_n}e_3 \big) (b^{(1)}+b^{(2)})  \varphi(x)\,dx\\
&= \frac{1}{h_n \eps_n (\nu_n-1)} e_3^T \int_\Omega \big( \bar{\nabla}_n^{\rm 2dim} \bar{y}_n \big( x - \frac{\eps_n}{2h_n}e_3 \big) \\ 
&\qquad\qquad\qquad\qquad\qquad - \nabla_n'  \bar{y}_n \big( x - \frac{\eps_n}{2h_n}e_3 \big) Z^{\rm 2dim} \big) (b^{(1)}+b^{(2)})  \varphi(x)\,dx\\
&= \frac{1}{(\nu_n-1)\eps_n} \int_\Omega\frac{(\bar{y}_n)_3( x - \frac{\eps_n}{2h_n}e_3)}{h_n}   \big((\bar{\nabla}_n^{\rm 2dim})^\ast \varphi(x) \\ 
&\qquad\qquad\qquad\qquad\qquad\qquad\qquad\qquad + \nabla_n'  \varphi(x) Z^{\rm 2dim} \big) (b^{(1)}+b^{(2)})  \,dx.
\end{align*}
Now,
\begin{align*}
  &\frac{1}{\eps_n} \Big((\bar{\nabla}_n^{\rm 2dim})^\ast \varphi(x) + \nabla_n'  \varphi(x) Z^{\rm 2dim} \Big) \\ 
&~~ \to \Big(\tfrac{1}{2}\nabla'^2 \varphi (x) [(z^i)',(z^i)']  - \tfrac{1}{8} \sum_{j=1}^4 \nabla'^2 \varphi (x) [(z^j)',(z^j)'] \Big)_{i=1,...,4} 
\end{align*}
uniformly. Therefore, \eqref{eq:hat-v-conv} gives
\begin{align}\label{eq:strainpart23infty}
\frac{1}{h_n^2} e_3^T \int_\Omega \bar{\nabla}_n^{\rm 2dim} \bar{y}_n \big( x - \frac{\eps_n}{2h_n}e_3 \big) (b^{(1)}+b^{(2)})  \varphi(x)\,dx\to 0,
\end{align}
if $\nu_n \to \infty$. For $\nu_n = \nu$ constant however, using \eqref{eq:hat-v-conv} and \eqref{eq:hat-v-form} we find
\begin{align}
&\frac{1}{h_n^2} e_3^T \int_\Omega \bar{\nabla}_n^{\rm 2dim} \bar{y}_n \big( x - \frac{\eps_n}{2h_n}e_3 \big) (b^{(1)}+b^{(2)})  \varphi(x)\,dx \notag \\
&~~\to \frac{1}{(\nu-1)} \int_\Omega\hat{v} \big(x', \frac{\lfloor (\nu-1) x_3 \rfloor}{\nu-1}  \big) \Big(\tfrac{1}{2}\nabla'^2 \varphi (x) [(z^i)',(z^i)']\Big)_{i=1,...,4} (b^{(1)}+b^{(2)})  \,dx \notag \\
&~~=\frac{1}{(\nu-1)} \int_\Omega \Big(\tfrac{1}{2}\nabla'^2 v(x')  [(z^i)',(z^i)']\Big)_{i=1,...,4} (b^{(1)}+b^{(2)}) \varphi (x) \,dx, \label{eq:strainpart23finite}
\end{align}
where we have used that $\sum_{i=1}^8 b_i = 0$.

We still need to find the limit of \eqref{eq:strainidentification1}. For any test function $\varphi \in C_c^\infty(\Omega; \R^3)$ we find
\begin{align*}
  &\int_\Omega \frac{1}{h_n^2} \Big( \bar{\nabla}_n^{\rm 2dim} \bar{y}_n \big( x + \frac{\eps_n}{2h_n}e_3 \big) - \bar{\nabla}_n^{\rm 2dim} \bar{y}_n \big( x - \frac{\eps_n}{2h_n}e_3 \big) \Big) b^{(2)} \cdot \varphi(x)\,dx\\ 
  &~~= \frac{\eps_n}{h_n} \int_\Omega \frac{1}{\eps_n h_n} \Big( \bar{y}_n \big( x + \frac{\eps_n}{2h_n}e_3) -  \bar{y}_n \big( x - \frac{\eps_n}{2h_n} e_3 \big) \Big) 
       \cdot (\bar{\nabla}_n^{\rm 2dim})^{*} P_n \varphi(x)b^{(2)}\,dx\\
  &~~= \frac{1}{h_n^2} \int_\Omega \dashint_{ \tilde{Q}_n(x)} \Big( \bar{y}_n \big( \xi + \frac{\eps_n}{2h_n}e_3) -  \bar{y}_n \big( \xi - \frac{\eps_n}{2h_n} e_3 \big) \Big) \, d\xi 
       \cdot (\bar{\nabla}_n^{\rm 2dim})^{*} P_n \varphi(x)b^{(2)}\,dx\\
  &~~= \frac{\eps_n}{h_n^3} \int_\Omega \partial_3 \tilde{y}_n (x) \cdot (\bar{\nabla}_n^{\rm 2dim})^{*} P_n \varphi(x)b^{(2)}\,dx\\
  &~~= \frac{\eps_n}{h_n} \int_\Omega P_n A_n (x) e_3 \cdot (\bar{\nabla}_n^{\rm 2dim})^{*} \varphi(x)b^{(2)}\,dx. 
\end{align*} 
Here the penultimate step is true by our specific choice of interpolation to define $\tilde{y}_n$, whereas the last step follows from \eqref{eq:grad-yn-Id-A} and $\bar{\nabla}_n^{\rm 2dim} \frac{1}{h_n} e_3 = 0$. If $\nu_n \to \infty$ this converges to $0$. In case $\nu_n = \nu$ constant we obtain from \eqref{eq:grad-yn-Id-A}
\begin{align}
  &\lim_{n \to \infty} \int_\Omega \frac{1}{h_n^2} \Big( \bar{\nabla}_n^{\rm 2dim} \bar{y}_n \big( x + \frac{\eps_n}{2h_n}e_3 \big) - \bar{\nabla}_n^{\rm 2dim} \bar{y}_n \big( x - \frac{\eps_n}{2h_n}e_3 \big) \Big) b^{(2)} \cdot \varphi(x)\,dx \notag \\ 
  &~~= - \frac{1}{\nu-1} \int_\Omega P A (x) e_3 \cdot \nabla'\varphi(x) Z^{\rm 2dim} b^{(2)}\,dx \notag \\
  &~~ = \frac{1}{\nu-1} \int_\Omega (\partial_1v(x'), \partial_2 v(x'), 0) \nabla'\varphi(x) Z^{\rm 2dim} b^{(2)}\,dx \notag \\
  &~~ = - \frac{1}{\nu-1} \int_\Omega \begin{pmatrix} \nabla'^2 v(x') Z^{\rm 2dim} b^{(2)} \\ 0 \end{pmatrix} \cdot \varphi(x) \,dx. \label{eq:strainpart1}
\end{align}

Summarizing \eqref{eq:strainpart212}, \eqref{eq:strainpart23infty}, \eqref{eq:strainpart23finite}, and \eqref{eq:strainpart1}, we see that for non-affine $b$ we have $\bar{G} b = 0$ in case $\nu_n \to \infty$ and 
\begin{align*}
  \bar{G}b 
  &= \begin{pmatrix}
-\frac{1}{\nu-1}\nabla'^2 v(x') Z^{\rm 2dim} b^{(2)} \\
\frac{1}{\nu-1} \sum_{i=1}^4 \tfrac{1}{2}\nabla'^2 v(x')  [(z^i)',(z^i)'] (b^{(1)} + b^{(2)})_i
\end{pmatrix} \\ 
  &= \begin{pmatrix}
-\frac{1}{2(\nu-1)}\nabla'^2 v(x') Z^{\rm 2dim} (b^{(2)} - b^{(1)}) \\
\frac{1}{2(\nu-1)} \sum_{i=1}^8 \nabla'^2 v(x')  [(z^i)',(z^i)'] b_i 
  \end{pmatrix} 
  - \frac{1}{8(\nu-1)} \Delta v(x') \sum_{j=1}^8 b_j e_3  
\end{align*}
as $\sum_{j=1}^8 b_j = 0$, if $\nu_n \equiv \nu$. 

Elementary computations show that for the affine basis vectors $b^k$, $k \in \{0,1,2,3\}$, 
\begin{align*}
  Z^{\rm 2dim} ((b^k)^2 - (b^k)^1) = 0 
\end{align*}
and also
\begin{align*}
  \sum_{i=1}^8 \nabla'^2 v(x')  [(z^i)',(z^i)'] b^k_i 
  - \frac{1}{4} \Delta v(x') \sum_{j=1}^8 b^k_j 
  = 0. 
\end{align*}
%
%
Thus combining with \eqref{eq:Gbar-affine}, for every $b \in \R^8$ we get 
\[ \bar{G}b = G Z b \] 
if $\nu_n \to \infty$ and 
\[ \bar{G}b 
   = P G Z b + \begin{pmatrix}
-\frac{1}{2(\nu-1)}\nabla'^2 v(x') Z^{\rm 2dim} (b^{(2)} - b^{(1)}) \\
\frac{1}{2(\nu-1)} \sum_{i=1}^8 \nabla'^2 v(x')  [(z^i)',(z^i)'] b_i 
  \end{pmatrix} 
  - \frac{1}{8(\nu-1)} \Delta v(x') \sum_{j=1}^8 b_j e_3. \]
if $\nu_n = \nu$ is constant. So $\bar{G} = GZ$ if $\nu_n \to \infty$ and 
\begin{align*} 
  \bar{G} 
  = P G Z &-\frac{1}{2(\nu-1)} \begin{pmatrix}
     \nabla'^2 v(x')  0 \\ 0 & 0 \end{pmatrix} Z_-   + \frac{1}{2(\nu-1)} e_3 \otimes ( \nabla'^2 v(x')  [(z^i)',(z^i)'] )_{i=1,\ldots,8} \\ 
  &\quad  - \frac{1}{8(\nu-1)} \Delta v(x')) e_3 \otimes (1, \ldots, 1). 
\end{align*} 
with $Z_-$ as in \eqref{eq:Zminus-def} if $\nu_n = \nu$ is constant. Noting that 
\[ \nabla'^2 v(x')  [(z^i)',(z^i)'] 
   = \begin{cases} 
        \frac{1}{4} (\partial_{11} v(x') + 2 \partial_{12} v(x') + \partial_{22} v(x')) & \text{if } i \in \{1,3,5,7\}, \\ 
        \frac{1}{4} (\partial_{11} v(x') - 2 \partial_{12} v(x') + \partial_{22} v(x')) & \text{if } i \in \{2,4,6,8\}, \\ 
     \end{cases} \] 
with $M$ as in \eqref{eq:M-def} this can be written as 
\begin{align*} 
  \bar{G} 
  = P G Z -\frac{1}{2(\nu-1)} \begin{pmatrix}
     \nabla'^2 v(x') & 0 \\ 0 & 0 \end{pmatrix} Z_- 
    + \frac{1}{2(\nu-1)} \partial_{12} v(x') M. 
\end{align*}
Last, we note that subsequences were indeed not necessary, as the limit is characterized uniquely.
\end{proof}

Having established convergence of the strain, the $\liminf$ inequality in Theorems~\ref{thm:Gammalimit}, \ref{thm:Gammalimit2} and \ref{thm:Gammalimit3} can now be shown by a careful Taylor expansion of $W(x, \cdot)$, cf.\ \cite{FJM:02,FJM:06,Sch:06}. 

\begin{proof}[Proof of the $\liminf$ inequality in Theorems~\ref{thm:Gammalimit}, \ref{thm:Gammalimit2} and \ref{thm:Gammalimit3}] 
 The $\liminf$ inequality in Theorem~\ref{thm:Gammalimit3} is an immediate consequence of the $\liminf$ inequality in Theorem~\ref{thm:Gammalimit} applied to a cell energy $W_{\rm cell}'$ of the form 
\[ W_{\rm cell}'(A) 
   = \begin{cases} 
        W_{\rm cell}(A), & \text{if } \dist(A, \SO(3) Z) < \delta, \\ 
        \dist^2(A, \SO(3) Z),  & \text{if } \dist(A, \SO(3) Z) \ge \delta. 
     \end{cases} \] 
Furthermore, in view of Proposition~\ref{prop:limiting-force} it suffices to establish the lower bound for $f_n = 0$.

Assume that $(y_n)$ is a sequence of atomistic deformations such that 
\[ \sup_n E_n (y_n) < \infty \] 
so that by Proposition~\ref{prop:energy-estimates} its modification and interpolation $(\tilde{y}_n)$ verifies the assertions of Theorem \ref{thm:contcomp}. Set 
\begin{equation*}
\bar{G}_n := \frac{1}{h_n^2} (R_n^T \bar{\nabla}_n \bar{y}_n - Z). 
\end{equation*}

By frame indifference and nonnegativity of the cell energy we have 
\begin{align*}
  &\frac{\eps_n^3}{h_n^5} E_n(y_n) 
  \ge \frac{\eps_n^3}{h_n^5} \sum_{ x \in (\tilde{\Lambda}_n')^{\circ}} W((x', h_n x_3), \bar{\nabla}_n \bar{y}_n(x))  \\ 
  &~~= \frac{1}{h_n^4} \int_{\Omega^{\rm in}_n} W \big(  \eps_n ( \lfloor \tfrac{x_1}{\eps_n} \rfloor + \tfrac{1}{2}, \lfloor \tfrac{x_2}{\eps_n} \rfloor + \tfrac{1}{2}, \lfloor \tfrac{h_n x_3}{\eps_n} \rfloor + \tfrac{1}{2}), Z + h_n^2 \bar{G}_n(x) \big) \, dx. 
\end{align*}

First assume that $\nu_n \to \infty$ as $n \to \infty$. Due to nonnegativity of $W_{\rm surf}$ we can estimate 
\begin{align*}
  \frac{\eps_n^3}{h_n^5} E_n(y_n) 
  &\ge \frac{1}{h_n^4} \int_{\Omega} \chi_{n}(x) W_{\rm cell} (Z + h_n^2 \bar{G}_n(x)) \, dx \\ 
  &= \int_{\Omega} \frac{1}{2} Q_{\rm cell} \big( \chi_{n}(x) \bar{G}_n(x) \big) 
     - h_n^{-4} \chi_n(x) \omega \big( | h_n^2 \bar{G}_n(x) | \big) \, dx, 
\end{align*}
where $\chi_n$ is the characteristic function of $\{ x \in \Omega^{\rm in}_n : \bar{G} \le h_n^{-1} \} \subset \Omega$ and 
\[ \omega(t) 
   := \sup \big\{ | \tfrac{1}{2} Q_{\rm cell}(F) - W_{\rm cell}(Z + F) | : F \in \R^{3 \times 8} \text{ with } |F| \le t \big\} \] 
so that $t^{-2} \omega(t) \to 0$ as $t \to 0$. Since $\bar{G}_n^2$ is bounded in $L^1(\Omega; \R^{3 \times 8})$ and $\chi_n (h_n^2 \bar{G}_n)^{-2} \omega(h_n^2 \bar{G}_n)$ converges to $0$ uniformly, 
\[ h_n^{-4} \chi_n \omega \big( h_n^2 \bar{G}_n \big) 
   = \bar{G}_n^2 \chi_n (h_n^2 \bar{G}_n)^{-2} \omega(h_n^2 \bar{G}_n)
   \to 0 \text{ in } L^1(\Omega; \R^{3 \times 8}). \] 
Moreover, $\chi_n \to 1$ boundedly in measure and so by Proposition \ref{prop:limiting-strain} $\chi_{n} \bar{G}_n \wto \bar{G} = GZ$, where $G$ satisfies \eqref{eq:cont-strain-identify-i} and \eqref{eq:cont-strain-identify-ii}. By lower semicontinuity it follows that 
\begin{align*}
  \liminf_{n \to \infty} \frac{\eps_n^3}{h_n^5} E_n(y_n) 
  &\ge \frac{1}{2} \int_{\Omega} Q_{\rm cell} \big( \bar{G}(x) \big) \, dx 
   \ge \frac{1}{2} \int_{\Omega} Q_{\rm cell}^{\rm rel} \big( \bar{G}(x) \big) \, dx \\ 
  &= \frac{1}{2} \int_{\Omega} Q_{\rm cell}^{\rm rel} \bigg( \begin{pmatrix} G_1(x') + (x_3 - \frac{1}{2}) G_2(x') & 0 \\ 0 & 0 \end{pmatrix} Z \bigg) \, dx. 
\end{align*}
Integrating the last expression over $x_3 \in (0,1)$ and noting that the integral of the cross terms vanish we obtain 
\begin{align*}
  \liminf_{n \to \infty} \frac{\eps_n^3}{h_n^5} E_n(y_n) 
  &\ge E_{\rm vK}(u,v). 
\end{align*}

Now suppose that $\nu_n \equiv \nu \in \N$. We let $\chi_n$ as above but now define 
\begin{align*}
   \omega(t) 
   &:= \sup \big\{ | \tfrac{1}{2} Q_{\rm cell}(F) - W_{\rm cell}(Z + F) | : 
           F \in \R^{3 \times 8} \text{ with } |F| \le t \big\} \\ 
   &\quad + 2 \sup \big\{ | \tfrac{1}{2} Q_{\rm surf}(F) - W_{\rm surf}(Z^{(1)} + F) | : 
           F \in \R^{3 \times 4} \text{ with } |F| \le t \big\}
\end{align*}
so that still $t^{-2} \omega(t) \to 0$ as $t \to 0$. With $\bar{G}(x) = (\bar{G}^{(1)}(x), \bar{G}^{(2)}(x))$ we have
\begin{align*}
  &\liminf_{n \to \infty} \frac{\eps_n^3}{h_n^5} E_n(y_n) 
  \ge \frac{1}{2} \int_{\Omega} Q_{\rm cell} \big( \bar{G}(x) \big) \, dx \\ 
  &\qquad+ \frac{1}{2(\nu-1)} \int_S Q_{\rm surf} \big( \bar{G}^{(1)}(x', \tfrac{1}{2(\nu-1)}) \big) + Q_{\rm surf} \big( \bar{G}^{(2)}(x', \tfrac{2\nu-3}{2\nu-2}) \big) \, dx, 
\end{align*}
where we have used that $\bar{G}$ is constant on $S \times (0,\frac{1}{\nu-1})$ and on $S \times (\frac{\nu-2}{\nu-1},1)$. Here (see Eq.~\eqref{eq:G3def} for $G_3$), 
\begin{align*}
  \bar{G}^{(1)}(x', \tfrac{1}{2\nu-2}) 
  &= \dashint_0^{\frac{1}{\nu-1}} G(x',x_3) \, dx_3 \, Z^{(1)} + \tfrac{1}{2(\nu-1)} G_3^{(1)}(x'), \\ 
  \bar{G}^{(2)}(x', \tfrac{2\nu-3}{2\nu-2}) 
  &= \dashint_{\frac{\nu-2}{\nu-1}}^1 G(x',x_3) \, dx_3 \, Z^{(2)} + \tfrac{1}{2(\nu-1)} G_3^{(2)}(x'). 
\end{align*}

The bulk part is estimated as 
\begin{align*}
  &\frac{1}{2} \int_{-\frac{1}{2}}^{\frac{1}{2}} Q_{\rm cell} \big( \bar{G}(x) \big) \, dx_3 \\ 
  &~~\ge \frac{1}{2(\nu-1)} \sum_{k = 1}^{\nu-1} Q_{\rm cell}^{\rm rel} 
       \bigg( \begin{pmatrix} \sym ( PG'' ) (x', \tfrac{2k - 1}{2\nu-2}) & 0 \\ 0 & 0 \end{pmatrix} Z + \tfrac{1}{2(\nu-1)} G_3(x') \bigg) \\ 
  &~~= \frac{1}{2(\nu-1)} \sum_{k = 1}^{\nu-1} Q_{\rm cell}^{\rm rel} 
       \bigg( \begin{pmatrix} \sym G_1(x') + \tfrac{2k - \nu}{2\nu-2} G_2(x') & 0 \\ 0 & 0 \end{pmatrix} Z + \tfrac{1}{2(\nu-1)} G_3(x') \bigg) \\ 
  &~~= \frac{1}{2(\nu-1)} \sum_{k = 1}^{\nu-1} \bigg[ Q_{\rm cell}^{\rm rel} 
       \bigg( \begin{pmatrix} \sym G_1(x') & 0 \\ 0 & 0 \end{pmatrix} Z + \tfrac{1}{2(\nu-1)} G_3(x') \bigg) \\ 
  &~~\qquad\qquad\qquad + \tfrac{(2k - \nu)^2}{(2\nu-2)^2}Q_{\rm cell}^{\rm rel} 
       \bigg( \begin{pmatrix} G_2(x') & 0 \\ 0 & 0 \end{pmatrix} Z\bigg) \bigg] \\ 
  &~~= \frac{1}{2} Q_{\rm cell}^{\rm rel} 
       \bigg( \begin{pmatrix} \sym G_1(x') & 0 \\ 0 & 0 \end{pmatrix} Z 
       + \tfrac{1}{2(\nu-1)} G_3(x') \bigg) + \tfrac{\nu(\nu-2)}{24(\nu-1)^2} Q_{\rm cell}^{\rm rel} 
       \bigg( \begin{pmatrix} G_2(x') & 0 \\ 0 & 0 \end{pmatrix} Z\bigg), 
\end{align*}
where we have used that $\sum_{k = 1}^{\nu-1}\tfrac{(2k - \nu)^2}{(2\nu-2)^3} = \frac{\nu(\nu-2)}{24(\nu-1)^2}$.

For the surface part first note that by \eqref{eq:Q-invariance}, for any $A = (a_{ij}) \in \R^{3 \times 3}$ and $B \in \R^{3 \times 4}$ we have 
\begin{align*}
  &Q_{\rm surf}(A Z^{(1)} + B) \\ 
  &~~= Q_{\rm surf} \big( A Z^{(1)} + B + (a_{3\cdot} \otimes e_3 - e_3 \otimes a_{3\cdot}) Z^{(1)} + (a_{\cdot 3} + a_{3 \cdot}) \otimes (1,1,1,1) \big) \\ 
  &~~= Q_{\rm surf} \bigg( \begin{pmatrix} A'' & 0 \\ 0 & 0 \end{pmatrix} Z^{(1)} + B \bigg) 
     = Q_{\rm surf} \bigg( \begin{pmatrix} \sym A'' & 0 \\ 0 & 0 \end{pmatrix} Z^{(1)} + B \bigg),  
\end{align*}
where $a_{\cdot 3}$ denotes the third column, $a_{3 \cdot}$ the third row and $A'' = (a_{ij})_{1 \le i,j \le 2}$ the upper left $2 \times 2$ part of $A$. Thus also 
\begin{align*}
  Q_{\rm surf}(A Z^{(2)} + B)  
  &= Q_{\rm surf} \big( A Z^{(1)} + a_{\cdot 3} \otimes (1,1,1,1) + B \big) \\ 
  &= Q_{\rm surf} \bigg( \begin{pmatrix} \sym A'' & 0 \\ 0 & 0 \end{pmatrix} Z^{(1)} + B \bigg).  
\end{align*}
It follows that 
\begin{align*}
  &Q_{\rm surf}(\bar{G}_1(x', \tfrac{1}{2\nu-2})) \\ 
  &~~= Q_{\rm surf} \bigg( \begin{pmatrix} \sym G_1(x') - \frac{\nu - 2}{2\nu - 2} G_2(x') & 0 \\ 0 & 0 \end{pmatrix} Z^{(1)} + \tfrac{1}{2(\nu-1)} G_3^{(1)}(x') \bigg) \\ 
  &~~= Q_{\rm surf} \bigg( \begin{pmatrix} \sym G_1(x') - \frac{1}{2} G_2(x') & 0 \\ 0 & 0 \end{pmatrix} Z^{(1)} + \frac{\partial_{12} v(x')}{4(\nu-1)} M^{(1)} \bigg), \\ 
  &Q_{\rm surf}(\bar{G}_2(x', \tfrac{2\nu - 3}{2\nu-2})) \\ 
  &~~= Q_{\rm surf} \bigg( \begin{pmatrix} \sym G_1(x') + \frac{\nu - 2}{2\nu - 2} G_2(x') & 0 \\ 0 & 0 \end{pmatrix} Z^{(1)} + \tfrac{1}{2(\nu-1)} G_3^{(2)}(x') \bigg) \\ 
  &~~= Q_{\rm surf} \bigg( \begin{pmatrix} \sym G_1(x') + \frac{1}{2} G_2(x') & 0 \\ 0 & 0 \end{pmatrix} Z^{(1)} + \frac{\partial_{12} v(x')}{4(\nu-1)} M^{(1)} \bigg) \bigg),   
\end{align*}
and so 
\begin{align*}
  &Q_{\rm surf}(\bar{G}_1(x', \tfrac{1}{2\nu-2})) + Q_{\rm surf}(\bar{G}_2(x', \tfrac{2\nu - 3}{2\nu-2})) \\ 
  &~~= 2 Q_{\rm surf} \bigg( \begin{pmatrix} \sym G_1(x') & 0 \\ 0 & 0 \end{pmatrix} Z^{(1)} 
          + \frac{\partial_{12} v(x')}{4(\nu-1)} M^{(1)}  \bigg) \\  
   &\qquad + \tfrac{1}{2} Q_{\rm surf} \bigg( \begin{pmatrix} G_2(x') & 0 \\ 0 & 0 \end{pmatrix} Z^{(1)}\bigg),   
\end{align*}

Adding bulk and surface contributions and integrating over $x'$ we arrive at 
\begin{align*}
  \liminf_{n \to \infty} \frac{\eps_n^3}{h_n^5} E_n(y_n) 
  &\ge \int_S \frac{1}{2} Q_{\rm cell}^{\rm rel} 
       \bigg( \begin{pmatrix} \sym G_1(x') & 0 \\ 0 & 0 \end{pmatrix} Z 
       + \tfrac{1}{2(\nu-1)} G_3(x') \bigg) \\ 
  &~~\qquad + \tfrac{\nu(\nu-2)}{24(\nu-1)^2} Q_{\rm cell}^{\rm rel} 
       \bigg( \begin{pmatrix} G_2(x') & 0 \\ 0 & 0 \end{pmatrix} Z\bigg) \\ 
  &~~\qquad + \tfrac{1}{\nu-1} Q_{\rm surf} \bigg( \begin{pmatrix} \sym G_1(x') & 0 \\ 0 & 0 \end{pmatrix} Z^{(1)} 
          + \frac{\partial_{12} v(x')}{4(\nu-1)} M^{(1)}  \bigg) \\ 
  &~~\qquad + \tfrac{1}{4(\nu-1)} Q_{\rm surf} \bigg( \begin{pmatrix} G_2(x') & 0 \\ 0 & 0 \end{pmatrix} Z^{(1)}\bigg) \, dx' \\ 
  &= E_{\rm vK}^{(\nu)}(u, v). \qedhere
\end{align*}
\end{proof}

\subsection{Upper bounds}

Without loss of generality we assume that $R^* = \Id$. (For general $R^*$ one just considers the sequence $R^* y_n$ with $y_n$ as in \eqref{eq:recovery-ansatz} and $R^*_n = R^*$ below. 

If $u : S \to \R^2$ and $v : S \to \R$ are smooth up to the boundary, we choose a smooth extension to a neighborhood of $S$ and define the lattice deformations $y_n : \tilde{\bar{\Lambda}}_n \to \R^3$ by restricting to $\tilde{\bar{\Lambda}}_n$ the mapping $y_n : \overline{\tilde{\Omega}_n^{\rm out}} \to \R^3$, defined by  
\begin{align}\label{eq:recovery-ansatz} 
  y_n(x) 
  = \begin{pmatrix} x' \\ h_n x_3 \end{pmatrix} 
     + \begin{pmatrix} h_n^2 u(x') \\ h_n v(x') \end{pmatrix} 
     - h_n^2 (x_3 - \tfrac{1}{2}) \begin{pmatrix} (\nabla' v(x'))^T \\ 0 \end{pmatrix} 
     + h_n^3 d(x', x_3) 
\end{align}
for all $x \in \overline{\tilde{\Omega}_n^{\rm out}}$. Here  $d : \overline{\tilde{\Omega}_n^{\rm out}} \to \R^3$ will be determined later, see \eqref{eq:d-choose-thick} and \eqref{eq:d-choose-thin} for films with many, respectively, a bounded number of layers. In both cases, $d$ is smooth and bounded in $W^{1,\infty}(\overline{\tilde{\Omega}_n^{\rm out}}; \R^3)$ uniformly in $n$. 

We let $R^*_n = \Id$ and $c_n = 0$ for all $n$ and define $\tilde{y}_n \in W^{1,2}(\tilde{\Omega}_n^{\rm out}; \R^3)$ as in \eqref{eq:yn-tilde-def} by interpolating as in Section~\ref{sec:preparations} (more precisely, descaling to $w_n$ and then interpolating and rescaling) to obtain $\tilde{y}_n=\dtilde{y}_n$. Analogously we let $\bar{y} = \dbar{y}$. We define $u_n$ and $v_n$ as in \eqref{eq:un-def} and \eqref{eq:vn-def}, respectively. It is straightforward to check that indeed $u_n \to u$ in $W^{1,2}(S; \R^2)$ and $v_n \to v$ in $W^{1,2}(S)$.

In order to estimate the energy of $y_n$ we need to compute its discrete gradient. Instead of directly calculating $\bar{\nabla} \bar{y}_n = (\bar{\partial}_1 \bar{y}_n, \ldots, \bar{\partial}_8 \bar{y}_n)$ it is more convenient to first determine $\bar{D} y_n = (\bar{D}_1 y_n, \ldots, \bar{D}_8 y_n)$ which for each $x \in \tilde{\Lambda}'_n - (\frac{\eps_n}{2}, \frac{\eps_n}{2}, \frac{\eps_n}{2 h_n})$ is defined by 
\begin{align*} 
  \bar{D}_i y_n(x) 
  &= \frac{1}{\eps_n} \big[ y_n \big( \hat{x} + \eps_n ( (a^i)', h_n^{-1} a^i_3) \big) - y_n (\hat{x}) \big], 
\end{align*}
where for $x \in \tilde{\Omega}_n^{\rm out}$ we have set 
\[ \hat{x} 
   = \big( \eps_n \lfloor \tfrac{x_1}{\eps_n} \rfloor, \eps_n \lfloor \tfrac{x_2}{\eps_n} \rfloor, \tfrac{\lfloor (\nu_n-1) x_3 \rfloor}{\nu_n-1} \big), \] 
so that $\tilde{Q}_n(x) = \hat{x} + (0,\eps_n)^2 \times (0, (\nu_n-1))$. We set  $a^i = \frac{1}{2}(1,1,1)^T + z^i \in \{0,1\}^3$ and write $A := (a^1, \ldots, a^8) = Z + \frac{1}{2}(1,1,1)^T \otimes (1, 1, 1, 1, 1, 1, 1, 1)^T$. Note that 
\begin{align}\label{eq:D-Nabla-Umrechnung} 
  \bar{D}_i y_n(x)
  = \bar{\partial}_i \bar{y}_n(x) - \bar{\partial}_1 \bar{y}_n(x) 
  \quad\mbox{and}\quad 
  \bar{\partial}_i \bar{y}_n(x) 
  = \bar{D}_i y_n(x) - \frac{1}{8} \sum_{j=1}^8 \bar{D}_j y_n(x). 
\end{align}
In particular, if $\bar{D} y_n(x)$ is affine, i.e., $\bar{D} y_n(x) = F A$ for some $F \in \R^{3 \times 3}$, then 
\begin{align}\label{eq:affine-Umrechnung} 
  \bar{\partial}_i \bar{y}_n(x) 
  &= F a^i - \frac{1}{8} \sum_{j=1}^8 F a^j 
   = F \Big( a^i - \frac{1}{2}(1,1,1)^T \Big) = F z^i
\end{align} 
and so $\bar{\nabla} \bar{y}_n(x) = F Z$.  

For $x$ in a fixed cell $\tilde{Q}_n(x) = \hat{x} + (0,\eps_n)^2 \times (0, (\nu_n-1))$, Taylor expansion of $y_n$ (restricted to $\overline{\tilde{Q}}_n(x)$) yields  
\begin{align*} 
  \bar{D}_i y_n(x) 
  &= \nabla' y_n(\hat{x}) (a^i)' + h_n^{-1} \partial_3 y_n(\hat{x}) a^i_3 + \frac{\eps_n}{2} (\nabla')^2 y_n(\hat{x}) [(a^i)', (a^i)'] \\ 
  &\qquad + \eps_n h_n^{-1} \sum_{j=1}^2 \partial_{j3} y_n(\hat{x}) a^i_j a^i_3 + \frac{\eps_n h_n^{-2}}{2} \partial_{33} y_n(\hat{x}) (a^i_3)^2 \\ 
  &\qquad + \frac{\eps_n^{2}}{6} \nabla^3 \big( (y_n)_1 (\zeta^1_{\eps_n}), (y_n)_2 (\zeta^2_{\eps_n}), (y_n)_2 (\zeta^2_{\eps_n}) \big)^T \\ 
  &\qquad \qquad \qquad [((a^i)', h_n^{-1} a^i_3), ((a^i)', h_n^{-1} a^i_3), ((a^i)', h_n^{-1} a^i_3)]  
\end{align*} 
for some $\zeta_{\eps_n} \in \hat{x} + [0,\eps_n]^2\times[0, \eps_n h_n^{-1}]$. Plugging in \eqref{eq:recovery-ansatz} we get  
\begin{align*} 
  \bar{D}_i y_n(x) 
  &= \bigg( \begin{pmatrix} \Id_{2\times2} \\ 0 \end{pmatrix} 
            + \begin{pmatrix} h_n^2 \nabla' u(\hat{x}') \\ h_n \nabla' v(\hat{x}') \end{pmatrix} \\ 
  &\qquad   - h_n^2 (\hat{x}_3 - \tfrac{1}{2}) \begin{pmatrix} \nabla' (\nabla' v(\hat{x}'))^T  \\ 0 \end{pmatrix} 
            + h_n^3 \nabla' d(\hat{x}) \bigg) (a^i)' \\ 
  &\qquad + h_n^{-1} \left( \begin{pmatrix} 0 \\ h_n \end{pmatrix} 
            + 0 - h_n^2 \begin{pmatrix} (\nabla' v(\hat{x}'))^T \\ 0 \end{pmatrix} 
            + h_n^3 \partial_3 d(\hat{x}) \right) a^i_3 \\ 
  &\qquad + \frac{\eps_n h_n}{2} \begin{pmatrix} 0 \\ (\nabla')^2 v(\hat{x}') [(a^i)', (a^i)'] \end{pmatrix} + O(\eps_n h_n^2) \\ 
  &\qquad - \eps_n h_n \begin{pmatrix} \nabla' (\nabla' v(\hat{x}'))^T \\ 0 \end{pmatrix} (a^i)' a^i_3 
          + O(\eps_n h_n^2) \\ 
  &\qquad + \frac{\eps_n h_n}{2} \partial_{33} d(\hat{x}) (a^i_3)^2 \\ 
  &\qquad + \frac{\eps_n^2}{6} \partial_{333} \big( d_1(\zeta^1_{\eps_n}), d_2(\zeta^2_{\eps_n}), d_3(\zeta^3_{\eps_n}) \big)^T (a^i_3)^3 
          + O(\eps_n^2 h_n). 
\end{align*}

It follows that 
\begin{align*} 
  \bar{D}_i y_n(x) 
  &= \bigg( \Id_{3\times3} + h_n \begin{pmatrix} h_n \nabla' u(\hat{x}') & - (\nabla' v(\hat{x}'))^T \\ \nabla' v(\hat{x}') & 0 \end{pmatrix} \\ 
    &\qquad - h_n^2 (\hat{x}_3 - \tfrac{1}{2}) \begin{pmatrix} (\nabla')^2 v(\hat{x}') & 0 \\ 0 & 0 \end{pmatrix} 
     + h_n^2 \begin{pmatrix} 0_{3\times2} & \partial_3 d(\hat{x}) \end{pmatrix}  \bigg) a^i \\ 
    &\qquad + \eps_n h_n \left( \begin{pmatrix} - (\nabla')^2 v(\hat{x}') (a^i)' a^i_3 \\ \frac{1}{2} (\nabla')^2 v(\hat{x}') [(a^i)', (a^i)'] \end{pmatrix} + \tfrac{1}{2} \partial_{33} d(\hat{x}) (a^i_3)^2 \right) \\ 
    &\qquad + \frac{\eps_n^2}{6} \partial_{333} \big( d_1(\zeta^1_{\eps_n}), d_2(\zeta^2_{\eps_n}), d_3(\zeta^3_{\eps_n}) \big)^T (a^i_3)^3 
     + O(\eps_n h_n^2 + \eps_n^2 h_n). 
\end{align*}

We define the skew symmetric matrix $B(\hat{x}) = B_n(\hat{x})$ by  
\begin{align*} 
  B(\hat{x}) 
  &= \begin{pmatrix} \frac{h_n^2}{2} ( \nabla' u(\hat{x}') - ( \nabla' u(\hat{x}'))^T ) & - h_n (\nabla' v(\hat{x}'))^T \\ h_n \nabla' v(\hat{x}) & 0 \end{pmatrix} \\ 
  &\qquad + \frac{h_n^2}{2} \begin{pmatrix} 0_{2\times2} & \partial_3 d'(\hat{x}) \\ - (\partial_3 d'(\hat{x}))^T & 0 \end{pmatrix}, 
\end{align*} 
where we have written $d' = (d_1, d_2)^T$ for $d = (d_1, d_2, d_3)^T$, and consider the special orthogonal matrix 
\begin{align*} 
  \e^{-B(\hat{x})}  
  &= \Id_{3\times3} - B(\hat{x}) + \frac{1}{2} B^2(\hat{x}) + O(|B(\hat{x})|^3) \\ 
  &= \Id_{3\times3} - h_n \begin{pmatrix} 0_{2\times2} & - (\nabla' v(\hat{x}'))^T \\ \nabla' v(\hat{x}') & 0 \end{pmatrix} \\ 
  &\qquad - \frac{h_n^2}{2} \begin{pmatrix} \nabla' u(\hat{x}') - ( \nabla' u(\hat{x}'))^T + \nabla' v(\hat{x}') \otimes \nabla' v(\hat{x}') & \partial_3 d'(\hat{x}) \\ - (\partial_3 d'(\hat{x}))^T & |\nabla' v(\hat{x}')|^2 \end{pmatrix} \\ 
  &\qquad + O(|h_n|^3).  
\end{align*} 
Now compute 
\begin{align}\label{eq:expBDy} 
\begin{split}
  &\e^{-B(\hat{x})} \bar{D}_i y_n(x) 
  = \bar{D}_i y_n(x) \\ 
  &\qquad   - h_n \begin{pmatrix} 0_{2\times2} & - (\nabla' v(\hat{x}'))^T \\ \nabla' v(\hat{x}') & 0 \end{pmatrix} 
       \bigg( \Id_{3\times3} 
         + h_n \begin{pmatrix} 0_{2\times2} & - (\nabla' v(\hat{x}'))^T \\ \nabla' v(\hat{x}') & 0 \end{pmatrix} \bigg) a^i \\ 
  &\qquad - \frac{h_n^2}{2} \begin{pmatrix} \nabla' u(\hat{x}') - ( \nabla' u(\hat{x}'))^T + \nabla' v(\hat{x}') \otimes \nabla' v(\hat{x}') & \partial_3 d'(\hat{x}) \\ - (\partial_3 d'(\hat{x}))^T & |\nabla' v(\hat{x}')|^2 \end{pmatrix} a^i \\ 
  &\qquad + O(h_n^3 + \eps_n h_n^2 + \eps_n^2 h_n) 
\end{split}\notag \\ 
\begin{split}
  &= \bigg( \Id_{3\times3} + h_n^2 \begin{pmatrix} {\rm sym} \nabla' u(\hat{x}') + \frac{1}{2} \nabla' v(\hat{x}') \otimes \nabla' v(\hat{x}')  & 0 \\ 0 & \frac{1}{2} |\nabla' v(\hat{x}')|^2 \end{pmatrix} \\ 
    &\qquad - h_n^2 (\hat{x}_3 - \tfrac{1}{2}) \begin{pmatrix} (\nabla')^2 v(\hat{x}') & 0 \\ 0 & 0 \end{pmatrix} 
     + h_n^2 \begin{pmatrix} 0_{2\times2} & \frac{1}{2} \partial_3 d'(\hat{x}) \\ \frac{1}{2} (\partial_3 d'(\hat{x}))^T & \partial_3 d_3(\hat{x}) \end{pmatrix}  \bigg) a^i \\
    &\qquad + \eps_n h_n \left( \begin{pmatrix} - (\nabla')^2 v(\hat{x}') (a^i)' a^i_3 \\ \frac{1}{2} (\nabla')^2 v(\hat{x}') [(a^i)', (a^i)'] \end{pmatrix} + \tfrac{1}{2} \partial_{33} d(\hat{x}) (a^i_3)^2 \right) \\ 
    &\qquad + \frac{\eps_n^2}{6} \partial_{333} \big( d_1(\zeta^1_{\eps_n}), d_2(\zeta^2_{\eps_n}), d_3(\zeta^3_{\eps_n}) \big)^T (a^i_3)^3 
     + O(h_n^3 + \eps_n h_n^2 + \eps_n^2 h_n). 
\end{split}
\end{align} 
Here, the error term is uniform in $\hat{x}$. 

We can now conclude the proof of Theorems~\ref{thm:Gammalimit}, \ref{thm:Gammalimit2} and \ref{thm:Gammalimit3}.

\begin{proof}[ Proof of the $\limsup$ inequality in Theorems~\ref{thm:Gammalimit}, \ref{thm:Gammalimit2} and \ref{thm:Gammalimit3}] 

As the discrete gradient $\bar{\nabla}_n \bar{y}_n$ is uniformly close to $\SO(3) Z$, the following arguments apply to show that $y_n$ defined by \eqref{eq:recovery-ansatz} serves as a recovery sequence in all three theorems. Moreover, in view of Proposition~\ref{prop:limiting-force} it suffices to construct recovery sequences for $f_n = 0$. 
 
We first specialize now to the case $\nu_n \to \infty$. For 
\begin{align}\label{eq:recov-G-def-thick} 
\begin{split}
  G(x) 
  &= G_1(x') + (x_3 - \tfrac{1}{2}) G_2(x') \\ 
  &= {\rm sym} \nabla' u(x') + \tfrac{1}{2} \nabla' v(x') \otimes \nabla' v(x') - (x_3 - \tfrac{1}{2}) (\nabla')^2 v(x'). 
\end{split}
\end{align}
choosing $d(x) = x_3 d_0(x') + \frac{x_3^2 - x_3}{2} d_1(x')$ with 
\begin{align}\label{eq:d-choose-thick} 
\begin{split}
  d_0(x') 
  &= \argmin_{b \in \R^3} Q_{\rm cell} \bigg[ \begin{pmatrix} G_1( x') & 0 \\ 0 &  \tfrac{1}{2} |\nabla' v(x')|^2 \end{pmatrix} Z + (b \otimes e_3) Z\bigg], \\ 
  d_1(x') 
  &= \argmin_{b \in \R^3} Q_{\rm cell} \bigg[ \begin{pmatrix} G_2(x') & 0 \\ 0 & 0 \end{pmatrix} Z + (b \otimes e_3) Z\bigg] 
\end{split}
\end{align} 
according to \eqref{eq:bmin-Q3}, from \eqref{eq:affine-Umrechnung} and \eqref{eq:expBDy}  we obtain 
\begin{align*} 
  \e^{-B(\hat{x})} \bar{\nabla} \bar{y}_n(x) 
  &= \bigg( \Id_{3\times3} + h_n^2 \begin{pmatrix} G(\hat{x}) & 0 \\ 0 & \frac{1}{2} |\nabla' v(\hat{x}')|^2 \end{pmatrix} \\ 
  &\qquad + h_n^2 {\rm sym} \big( ( d_0(\hat{x}) + (\hat{x}_3 - \tfrac{1}{2}) d_1(\hat{x}) ) \otimes e_3 \big) \bigg) Z 
     + O(h_n^3 + \eps_n h_n) 
\end{align*} 
and, Taylor expanding $W_{\rm cell}$, we see that due to the smoothness of $u$ and $v$ the piecewise constant mappings $x \mapsto h_n^{-4} W_{\rm cell} (\bar{\nabla}\bar{y}_n(x)) = h_n^{-4} W_{\rm cell} (e^{-B(\hat{x})} \bar{\nabla} \bar{y}_n(x))$ converge uniformly to 
\begin{align*} 
  \frac{1}{2} Q_{\rm cell}^{\rm rel} \bigg( \begin{pmatrix} G & 0 \\ 0 & 0 \end{pmatrix} Z \bigg) 
  = \frac{1}{2} Q_2(G). 
\end{align*} 
This shows that 
\begin{align*}
  \lim_{n \to \infty} h_n^{-4} E_n(y_n) 
  &= \frac{1}{2} \int_S Q_2(G(x)) \, dx \\ 
  &= \int_S \frac{1}{2} Q_2(G_1(x')) + \frac{1}{24} Q_2(G_2(x')) \, dx' 
  = E_{\rm vK}(u,v)    
\end{align*}
and thus finishes the proof in case $\nu_n \to \infty$. 
\medskip 

Now suppose that $\frac{\eps_n}{h_n} \equiv \frac{1}{\nu-1}$. Abbreviating $(\nabla')^2 v(\hat{x}') = -G_2(\hat{x}') = -G_2 = (f_{ij}) \in \R^{2 \times 2}$, we observe that 
\begin{align*} 
  &\begin{pmatrix}  2G_2 (a^i)' a^i_3 \\ -(a^i)'^T G_2 (a^i)' \end{pmatrix}_{i = 1, \ldots, 8} \\  
  &~~= \begin{pmatrix} 0 & 0 & 0 & 0 & 0 & -2f_{11} & -2f_{11} - 2f_{12} & -2f_{12} \\ 
                    0 & 0 & 0 & 0 & 0 & -2f_{21} & -2f_{21} - 2f_{22} & -2f_{22} \\ 
                    0 & f_{11} & \sum_{\mu,\nu} f_{\mu\nu} & f_{22} & 0 & f_{11} & \sum_{\mu,\nu} f_{\mu\nu} & f_{22}\end{pmatrix}, 
\end{align*} 
and hence, with $b = b(\hat{x}') = \big( (\partial_{11}+\partial_{12}) v(\hat{x}'), (\partial_{21}+\partial_{22}) v(\hat{x}'), 0 \big)^T = (f_{11} + f_{12}, f_{21} + f_{22}, 0)^T$, 
\begin{align*} 
  &\begin{pmatrix} 2G_2 (a^i)' a^i_3 \\ -(a^i)'^T G_2 (a^i)' \end{pmatrix}_{i = 1, \ldots, 8} 
  - (e_3 \otimes b - b \otimes e_3) A \\
  &\qquad = \begin{pmatrix} 0 & 0 & 0 & 0 & f_{11} + f_{12} & -f_{11} + f_{12} & -f_{11} - f_{12} & +f_{11} - f_{12} \\ 
                    0 & 0 & 0 & 0 & f_{21} + f_{22} & -f_{21} + f_{22} & -f_{21} - f_{22} & f_{21} - f_{22} \\ 
                    0 & -f_{12} & 0 & -f_{21} & 0 & -f_{12} & 0 & -f_{21}\end{pmatrix}, \\ 
  &\qquad = \begin{pmatrix} G_2 & 0 \\ 0 & 0 \end{pmatrix} ( Z + Z_-) 
             + \frac{1}{2} f_{12} \big( 2 M - e_3 \otimes (1, \ldots, 1) \big) \\ 
  &\qquad = \begin{pmatrix} G_2 & 0 \\ 0 & 0 \end{pmatrix} A - \frac{1}{2} b \otimes (1, \ldots, 1) 
            + \begin{pmatrix} G_2 & 0 \\ 0 & 0 \end{pmatrix} Z_- + \frac{f_{12}}{2} \big( 2 M - e_3 \otimes (1, \ldots, 1) \big). 
\end{align*} 
This shows that 
\begin{align*} 
  &\begin{pmatrix} -(\nabla')^2 v(\hat{x}') (a^i)' a^i_3 \\ \frac{1}{2} (\nabla')^2 v(\hat{x}') [(a^i)', (a^i)'] \end{pmatrix}_{i = 1, \ldots, 8} \\ 
  &~~= \frac{1}{2} \bigg( e_3 \otimes b - b \otimes e_3 + \begin{pmatrix} G_2 & 0 \\ 0 & 0 \end{pmatrix} \bigg) A 
   - \frac{1}{4} ( b + e_3 ) \otimes (1, \ldots, 1) \\ 
  &\qquad     + \frac{1}{2} \begin{pmatrix} G_2 & 0 \\ 0 & 0 \end{pmatrix} Z_- + \frac{1}{2} f_{12} M. 
\end{align*} 

We define the affine part of the strain $G(x) = G_1(x') + (x_3 - \tfrac{1}{2}) G_2(x')$ as in \eqref{eq:recov-G-def-thick}. The non-affine part is abbreviated by $\frac{1}{2(\nu - 1)} G_3(x')$ as in \eqref{eq:G3def}. Then using \eqref{eq:expBDy} we can write 
\begin{align*} 
  &\e^{-B(\hat{x})} \bar{\nabla} \bar{y}_n(x) \\  
  &= \bigg[ \Id_{3\times3} + h_n^2 \begin{pmatrix} G(\hat{x}',\hat{x}_3 + \tfrac{1}{2(\nu-1)}) & 0 \\ 0 & \frac{1}{2} |\nabla' v(\hat{x}')|^2 \end{pmatrix} 
     + h_n^2 {\rm sym} \big( \partial_3 d(\hat{x}) ) \otimes e_3 \big) \\ 
  &\qquad 
     + \frac{h_n^2}{2(\nu - 1)} \big( e_3 \otimes b(\hat{x}') - b(\hat{x}') \otimes e_3 \big) \bigg] Z 
     + \frac{h_n^2}{2(\nu - 1)}G_3(\hat{x}')  + O(h_n^3) \\ 
  &\qquad + \Big[ \frac{\eps_n h_n}{2} \partial_{33} d(\hat{x}) 
          + \frac{\eps_n^2}{6} \partial_{333} \big( d_1(\zeta^1_{\eps_n}), d_2(\zeta^2_{\eps_n}), d_3(\zeta^3_{\eps_n}) \big)^T \Big] \otimes (z^1_3, \ldots, z^8_3),   
\end{align*} 
where we have used \eqref{eq:affine-Umrechnung} and \eqref{eq:D-Nabla-Umrechnung}. 

We set 
\begin{align*} 
  d_0(x') 
  &= \argmin_{d \in \R^3} Q_{\rm cell} \bigg[ \begin{pmatrix} G_1(x') & 0 \\ 0 &  \tfrac{1}{2} |\nabla' v(x')|^2 \end{pmatrix} Z + {\rm sym} (d \otimes e_3) Z \\ 
  &\qquad \qquad \qquad \qquad \qquad \qquad \qquad \qquad \qquad + \frac{1}{2(\nu - 1)} G_3(x') \bigg], \\ 
  d_1(x') 
  &= \argmin_{d \in \R^3} Q_{\rm cell} \bigg[ \begin{pmatrix} G_2(x') & 0 \\ 0 & 0 \end{pmatrix} Z + {\rm sym} (d \otimes e_3) Z\bigg] 
\end{align*} 
according to \eqref{eq:bmin-Q3} and define $d : S' \times [0,1] \to \R$, $S'$ a neighborhood of $S$, inductively by $d(x,0) = 0$ and  
\begin{align}\label{eq:d-choose-thin} 
\begin{split} 
  d(x', \tfrac{j-1}{\nu-1} + t) 
  = d(x', \tfrac{j-1}{\nu-1}) + t d_0(x') + t \tfrac{2j - \nu}{2(\nu-1)} d_1(x') 
  \quad \mbox{if} \quad t \in [\tfrac{j-1}{\nu-1}, \tfrac{j}{\nu-1}], 
\end{split}
\end{align}
for $j = 1, \ldots, \nu-1$. Then $d$ is smooth in $x'$ and piecewise linear in $x_3$, more precisely, affine in $x_3$ in between two atomic layers: On $S' \times [\frac{j-1}{\nu-1}, \frac{j}{\nu-1}]$, $j \in \{ 1, \ldots, \nu-1 \}$, it satisfies 
\[ \partial_3 d(x) 
   = d_0(x') + \tfrac{2j-\nu}{2(\nu-1)} d_1(x') 
   = d_0(x') + (\hat{x}_3 - \tfrac{1}{2} + \tfrac{1}{2(\nu-1)}) d_1(x') \] 
since $\hat{x}_3 = \hat{x}_3(x) = \tfrac{j-1}{\nu-1}$. Taylor expanding $W_{\rm cell}$, we see that the piecewise constant mappings $x \mapsto h_n^{-4} W_{\rm cell} (\bar{\nabla} \bar{y}_n(x)) = h_n^{-4} W_{\rm cell} (e^{-B(\hat{x})} \bar{\nabla} \bar{y}_n(x))$ converge uniformly on $S' \times [\frac{j-1}{\nu-1}, \frac{j}{\nu-1}]$ to 
\begin{align*} 
  \frac{1}{2} Q_{\rm cell}^{\rm rel} \bigg( \begin{pmatrix} G_1(x') + \tfrac{2j-\nu}{2(\nu-1)} G_2(x') & 0 \\ 0 & 0 \end{pmatrix} Z 
   + \frac{1}{2(\nu - 1)} G_3(x')  \bigg) 
\end{align*} 
for each $j \in \{ 1, \ldots, \nu-1 \}$. Since $\tfrac{1}{\nu-1} \sum_{j = 1}^{\nu-1} \tfrac{2j-\nu}{2(\nu-1)} = 0$ and $\tfrac{1}{\nu-1} \sum_{j = 1}^{\nu-1} \big( \tfrac{2j-\nu}{2(\nu-1)} \big)^2 = \tfrac{\nu(\nu-2)}{12(\nu-1)^2}$, this shows  
\begin{align}\label{eq:bulk-contr} 
\begin{split}
  \frac{1}{h_n^{4}} \int_{\tilde{\Omega}_n^{\rm out}} W_{\rm cell} (\bar{\nabla}\bar{y}_n(x)) \, dx 
  &\to \int_S \frac{1}{2} Q_{\rm cell}^{\rm rel} \bigg( \begin{pmatrix} G_1(x') & 0 \\ 0 & 0 \end{pmatrix} Z 
   + \frac{1}{2(\nu - 1)} G_3(x') \bigg) \\ 
  &\qquad \qquad + \frac{\nu(\nu-2)}{24(\nu-1)^2} Q_{\rm cell}^{\rm rel} \bigg( \begin{pmatrix} G_2(x') & 0 \\ 0 & 0 \end{pmatrix} Z \bigg) \, dx'.  
\end{split}
\end{align} 

For the surface part we write $\bar{\nabla}\bar{y}_n = ([\bar{\nabla}\bar{y}_n]^{(1)}, [\bar{\nabla}\bar{y}_n]^{(2)})$ and use that the piecewise constant mappings $S' \times [0, \frac{1}{\nu-1}] \to \R$, \[x \mapsto h_n^{-4} W_{\rm surf} ([\bar{\nabla}\bar{y}_n(x)]^{(1)}) = h_n^{-4} W_{\rm surf} ([e^{-B(\hat{x})} \bar{\nabla}\bar{y}_n(x)]^{(1)}),\] converge uniformly to 
\begin{align*} 
  &\frac{1}{2} Q_{\rm surf} \bigg( \begin{pmatrix} G_1(x') - \tfrac{\nu - 2}{2(\nu-1)} G_2(x') & 0 \\ 0 & 0 \end{pmatrix} Z 
   + \frac{1}{2(\nu - 1)} G_3(x')  \bigg) \\
  &~~= \frac{1}{2} Q_{\rm surf} \bigg( \begin{pmatrix} \sym G_1(x') - \frac{1}{2} G_2(x) & 0 \\ 0 & 0 \end{pmatrix} Z^{(1)} + \frac{\partial_{12} v(x')}{4(\nu-1)} M^{(1)} \bigg).  
\end{align*} 
Similarly, the mappings $S' \times [\frac{\nu-2}{\nu-1}, 1] \to \R$, \[x \mapsto h_n^{-4} W_{\rm surf} ([\bar{\nabla}\bar{y}_n(x)]^{(2)}) = h_n^{-4} W_{\rm surf} ([e^{-B(\hat{x})} \bar{\nabla}\bar{y}_n(x)]^{(2)}),\] converge uniformly to 
\begin{align*} 
  \frac{1}{2} Q_{\rm surf} \bigg( \begin{pmatrix} \sym G_1(x') + \frac{1}{2} G_2(x) & 0 \\ 0 & 0 \end{pmatrix} Z^{(1)} + \frac{\partial_{12} v(x')}{4(\nu-1)} M^{(1)} \bigg). 
\end{align*} 
So with $S^{\rm out}_n$ such that $\tilde{\Omega}^{\rm out}_n = S^{\rm out}_n \times (0,1)$,  
\begin{align}\label{eq:surf-contr}
\begin{split} 
  &\frac{1}{h_n^{4}(\nu-1)} \int_{ S^{\rm out}_n} W_{\rm surf} \big( [\bar{\nabla}\bar{y}_n(x',\tfrac{1}{2(\nu-1)})]^{(1)} \big) + W_{\rm surf} \big( [\bar{\nabla}\bar{y}_n(x',\tfrac{2\nu - 3}{2(\nu-1)})]^{(2)} \big) \, dx' \\  
  &~~\to \int_S \tfrac{1}{\nu-1}Q_{\rm surf} \bigg( \begin{pmatrix} \sym G_1(x') & 0 \\ 0 & 0 \end{pmatrix} Z^{(1)} + \frac{\partial_{12} v(x')}{4(\nu-1)} M^{(1)} \bigg) \\ 
  &\qquad\qquad\qquad\qquad\qquad + \frac{1}{4(\nu-1)}Q_{\rm surf} \bigg( \begin{pmatrix} G_2(x) & 0 \\ 0 & 0 \end{pmatrix} Z^{(1)} \bigg) dx'. 
\end{split} 
\end{align} 

Summarizing \eqref{eq:surf-contr} and \eqref{eq:bulk-contr}, we have shown that 
\begin{align*} 
  \lim_{n \to \infty} h_n^{-4} E_n(y_n) 
  = \lim_{n \to \infty} \eps_n^3 h_n^{-5} \sum_{x \in \tilde{\Lambda}_n'} W \big( x, \bar{\nabla}y_n(x) \big) 
  = E_{\rm vK}^{(\nu)} (u, v) 
\end{align*} 
as $n \to \infty$, where we have also used that the contribution of the lateral boundary cells $\eps_n^3 h_n^{-5} \sum_{x \in \partial \tilde{\Lambda}_n'} W(x, \bar{\nabla}y_n(x))$ is negligible in the limit $n \to \infty$. 
\end{proof} 

\begin{proof}[Proof of the energy barrier in Theorem~\ref{thm:Gammalimit3}]
If a sequence of $w_n\in \mathcal{S}_{\delta}$ satisfies $E_n(w_n) \leq C h_n^4$, then the proof of Proposition~\ref{prop:energy-estimates} shows $\frac{\eps_n^3}{h_n} E_{\rm atom} (w_n) \leq C h_n^4$. Hence,
\[ \dist^2(\bar{\nabla} w_n(x),\SO(3)Z) 
   \le C E_{\rm atom} (w_n) 
   \le C h_n^5 \eps_n^{-3} 
   = C (\nu_n-1)^5 \eps_n^2, \]
which tends to $0$ by assumption. This implies that  $w_n \in \mathcal{S}_{\delta/2}$ for $n$ large enough.
\end{proof}


\section*{Acknowledgments.}

This work was supported by the Deutsche Forschungsgemeinschaft (DFG, German Research Foundation) under project number 285722765, as well as the Engineering and Physical Sciences Research Council (EPSRC) under the grant EP/R043612/1.

\bibliographystyle{alpha} 
\bibliography{papers}

\end{document}